\documentclass[a4paper]{article}

\usepackage[latin1]{inputenc}
\usepackage[T1]{fontenc}
\usepackage{amssymb,amsmath,amsthm,amscd,mathrsfs}
\usepackage[all]{xy}
\usepackage{color}
\usepackage{fullpage}

\newtheorem{thm}{Theorem}[section]
\newtheorem{defi}[thm]{Definition}
\newtheorem{prop}[thm]{Proposition}
\newtheorem{lemme}[thm]{Lemma}
\newtheorem{cor}[thm]{Corollary}

\newtheorem{nota}[thm]{Notation}
\theoremstyle{remark}

\newtheorem{rmk}[thm]{Remark}
\newtheorem{ex}[thm]{Example}

\DeclareMathOperator{\Z}{\mathbb{Z}}
\DeclareMathOperator{\sing}{Sing}
\DeclareMathOperator{\codim}{Codim}

\DeclareMathOperator{\rk}{rk}

\DeclareMathOperator{\Image}{Im}
\DeclareMathOperator{\discr}{discr}
\DeclareMathOperator{\Fix}{Fix}

\DeclareMathOperator{\tors}{torsion}
\DeclareMathOperator{\id}{id}

\DeclareMathOperator{\Sym}{Sym}
\DeclareMathOperator{\diag}{diag}
\DeclareMathOperator{\Ker}{Ker}
\DeclareMathOperator{\F}{\mathbb{F}_{p}}
\DeclareMathOperator{\Q}{\mathbb{Q}}
\DeclareMathOperator{\Bl}{Bl}

\newcommand{\eq}[1][r]
{\ar@<-3pt>@{-}[#1]
\ar@<-1pt>@{}[#1]|<{}="gauche"
\ar@<+0pt>@{}[#1]|-{}="milieu"
\ar@<+1pt>@{}[#1]|>{}="droite"
\ar@/^2pt/@{-}"gauche";"milieu"
\ar@/_2pt/@{-}"milieu";"droite"}

\newcommand{\incl}[1][r]
  {\ar@<-0.2pc>@{^(-}[#1] \ar@<+0.2pc>@{-}[#1]}

\begin{document}

\title{\bf On the integral cohomology of quotients of manifolds by cyclic groups}

\author{Grégoire \textsc{Menet}}

\maketitle
\begin{abstract}
We propose new tools based on basic lattice theory to calculate the integral cohomology 
of the quotient of a manifold by an automorphism group of prime order. As examples of applications, 
we provide the Beauville--Bogomolov forms of some primitively symplectic orbifolds; we also 
show a new expression for a basis of the integral cohomology of a Hilbert scheme of two points on a surface. 
\end{abstract}

\centerline{\small{AMS Classification 2010: 55N10, 14F43, 53C26}}

\section{Introduction}

Consider a compact connected orientable manifold (without boundary) $X$ and a finite automorphism group $G$, of prime order $p$, acting on $X$. 
In this article, we study the integral cohomology of the quotient $X/G$.
The group $G$ acts naturally on the cohomology of $X$.
In the case of cohomology with rational coefficients, the problem is much simpler,
$H^{*}(X/G,\mathbb{Q})$ is isomorphic to the invariant space $H^{*}(X,\mathbb{Q})^{G}$.
Nevertheless, for integral cohomology this property does not hold anymore.

A fundamental tool for studying this question is given by the work of Smith in \cite{Smith}.
Let $\pi:X\rightarrow X/G$ be the quotient map. Smith has constructed a push-forward map $\pi_*:H^{*}(X,\Z)\rightarrow H^{*}(X/G,\Z)$ such that:
\begin{equation}
\pi_{*}\circ\pi^{*}=p\id_{H^{*}(X/G,\Z)}, \ \ \ \ \ \ \pi^{*}\circ\pi_{*}=\sum_{g\in G}{g^{*}}.
\label{SmithIntro}
\end{equation}
Let $T$ be a $\Z$-module, we denote by $T_f:=\frac{T}{\tors}$ the torsion-free part of $T$.
Then, (\ref{SmithIntro}) provides the exact sequence
\begin{equation}
\xymatrix{ 0\ar[r]&\pi_{*}(H^{k}(X,\Z))_f\ar[r] & H^{k}(X/G,\Z)_f\ar[r] & (\Z/p\Z)^{\alpha_{k}(X)}\ar[r]& 0,}
\label{MainExactSequence}
\end{equation}
where $\alpha_{k}(X)$ is a non-negative integer 
which we call the \emph{$k$-th coefficient of surjectivity of $X$} (See Section \ref{DefinitionIntro}).

The objective of this paper is to provide tools for calculating $\alpha_{k}(X)$. Using basic lattice theory,
we provide an upper bound for the coefficients of surjectivity. 
For a finer result, one has to study the local action of $G$ around the fixed points.
A \emph{simple fixed point} is a fixed point such that
the local action of $G$ around is given by a matrix of the form 
$\diag(1,...,1,\xi_p^d,...,\xi_p^d)$, where $\xi_p$ is a $p$-th root of unity with $d\in\lbrace1,...,p-1\rbrace$. These fixed points have the interest to provide singularities on $X/G$ that can be solved with only one blow-up.
For instance, all fixed points are simple when $p=2$; we will see other examples of such fixed points in Section \ref{application}. 
Our main theorem is:
\begin{thm}\label{Maincor}
Let $X$ be a compact K\"ahler manifold of complex dimension $n$ and $G$ an automorphism group of prime order $p\leq 19$. Let $c:=\codim \Fix G$. 
Assume that:
\begin{itemize}
\item[(i)]
$H^*(X,\Z)$ and  $H^*(\Fix G,\Z)$ are $p$-torsion-free,
\item[(ii)]
the spectral sequence of equivariant cohomology with coefficients in $\F$ degenerates at the second page,
\item[(iii)]
all fixed points of $G$ are simple,
\item[(iv)]
$c\geq\frac{n}{2}+1$.
\end{itemize}
Then, for all $2n-2c+1\leq k \leq2c-1$, the $k$-th coefficient of surjectivity $\alpha_k(X)$ of $X$ vanishes.

\end{thm}
Condition (ii) is natural in the sense that there is already a theory
to show the degeneration of spectral sequences at the second page (see for instance \cite{Deligne}). Sarti, Boissi\`ere and Niper-Wisskirchen, in \cite{SmithTh},
also developed theorems of degeneration of the equivariant spectral sequence that we will use in our applications.
We will also see (Theorem \ref{main}) that condition (ii) can be replaced by a numerical condition involving invariants related to the action of the group and the cohomology of the fixed locus. 

This result can be applied to compute the Beauville--Bogomolov form of some primitively symplectic orbifolds. 
A \emph{orbifold} is a compact analytic complex space with
at worst finite quotient singularities. 
A compact Kähler orbifold is called \emph{symplectic} if its singularities are in codimension 4 and its smooth locus
is endowed with an everywhere non-degenerate holomorphic 2-form.
In addition, a symplectic orbifold is said \emph{primitively symplectic}\footnote{Here, we prefer the terminology \emph{primitively symplectic}
(introduced for the first time by Fujiki in \cite{Fujiki}) rather than \emph{irreducible symplectic} which we reserve for varieties appearing in 
a Bogomolov type decomposition theorem.} if 
the holomorphic 2-form is unique up to scaling. Such varieties are good candidates to generalize the short
list of known irreducible symplectic manifolds. Indeed, some aspects of the theory were already generalized in \cite{Namikawa} and \cite{Mat},
for instance, the Beauville--Bogomolov form, the local Torelli theorem, and the Fujiki formula. 
We recall that the Fujiki formula (\ref{Fujiki}) establishes a proportionality between the cup-product and the Beauville--Bogomolov form; 
the coefficient of proportionality is called the \emph{Fujiki constant}. 

In this paper we compute the Beauville--Bogomolov form of a manifold of $K3^{[2]}$-type quotiented by symplectic automorphisms
of order 3 and 11. 
There are two different examples of symplectic automorphisms $\sigma$ of order 11 on a manifold of $K3^{[2]}$-type $X$ (Example 4.5.1 and Example 4.5.2 in \cite{MongT}).
As explained in \cite[Section 7.4.4]{MongT}, the lattice $H^{2}(X,\Z)^{G}$ will be either $\begin{pmatrix} 6 & 2 & 2 \\ 2 & 8 & -3\\ 2 & -3 & 8\end{pmatrix}$ or $\begin{pmatrix} 2 & 1 & 0 \\ 1 & 6 & 0\\ 0 & 0 & 22\end{pmatrix}$.
We denote the quotients $X/\sigma$, respectively, by $M_{11}^{1}$ and $M_{11}^{2}$.
\begin{thm}\label{MainOrder11}
Let $X$ be a manifold of $K3^{[2]}$-type. Let $G$ be a symplectic automorphism group of order 11 acting on $X$. Then, the Fujiki constant of both $M_{11}^{1}$ and $M_{11}^{2}$ is $33$ and the Beauville--Bogomolov lattices are:
$$H^2(M_{11}^{1},\Z)\simeq\begin{pmatrix} 2 & -1 & 3 \\ -1 & 8 & -1\\ 3 & -1 & 6\end{pmatrix},\ 
H^2(M_{11}^{2},\Z)\simeq\begin{pmatrix} 2 & 1 & 0 \\ 1 & 6 & 0\\ 0 & 0 & 2\end{pmatrix}.$$
\end{thm}
In \cite[Theorem 7.2.7]{MongT}, Mongardi distinguishes two kinds of symplectic automorphism of order 3 on a $K3^{[2]}$-type manifold one with 27 isolated fixed points
and another with a fixed abelian surface.
\begin{thm}\label{MainOrder3}
Let $X$ be a manifold of $K3^{[2]}$-type. 
Let $G$ be a symplectic automorphisms group of order 3 of $X$ with 27 isolated fixed points. We denote $M_{3}=X/G$. 
Then, the Beauville--Bogomolov lattice $H^2(M_{3},\Z)$ is isomorphic to $U(3)\oplus U^{2}\oplus A_{2}^{2}\oplus (-6)$, and the Fujiki constant of $M_{3}$ is $9$.
\end{thm}
These examples provide new small-dimensional moduli spaces of singular primitively symplectic varieties. 
In particular, $M_{11}^{1}$ and $M_{11}^{2}$ provide two positive definite Beauville--Bogomolov lattices of rank 3.
It follows that the period domains of $M_{11}^{i}$, $i\in\{1,2\}$ are given by the zero of the Beauville--Bogomolov quadratic forms on $\mathbb{P}(H^2(M_{11}^{i},\mathbb{C}))\simeq \mathbb{P}^2$, 
hence are conics on $\mathbb{P}^{2}$. 
Moreover, none of their deformations admit a Lagrangian fibration, indeed
the existence of a Lagrangian fibration is in contradiction with having a definite positive Beauville--Bogomolov form since, 
by the Fujiki formula,
the form vanishes on
the pull-back of an ample divisor in the base of the fibration. 
Hence, it seems that
the varieties $M_{11}^{1}$ and $M_{11}^{2}$ look very different 
from the known smooth irreducible symplectic varieties.
So, $M_{11}^{1}$ and $M_{11}^{2}$ could be interesting examples to study in order 
to develop the theory of singular primitively symplectic varieties. 

As it is well known, 
the Hilbert scheme of two points $A^{[2]}$ on a surface $A$ can be obtained as a quotient of 
$\Bl_{\Delta}(A\times A)$ the blow-up of $A\times A$ in the diagonal $\Delta$. 
Therefore, studying the integral cohomology of $A^{[2]}$ is a possible application of our 
theory. As another application of our main theorem, we provide 
a new expression of a basis of the cohomology of $A^{[2]}$, already studied in \cite{QinWang} and \cite{Totaro}. This basis has the advantage to only require that the surface is Kähler. Moreover, the basis is expressed in terms of pull-backs and push-forwards of classes of the surface, allowing to calculate easily the ring structure of $H^*(A^{[2]},\Z)$ (see Lemma \ref{IntersectionHilb2}).

This paper generalizes results of the author's Ph.D. thesis \cite{Lol} (see also \cite{Lol2}), 
and it was already applied in \cite{Kapfer} for finding the Beauville--Bogomolov form of a partial resolution of the quotient by a symplectic involution of the generalized Kummer of dimension 4 which is the first example of odd Beauville--Bogomolov form.

\

The paper is structured as follows. 
Since the integral cohomology has a lattice structure, in Section \ref{BasicLatticeZModule}, we study a general lattice $T$ endowed with an isometry group $G$ of prime order.
In particular, using the $\Z[G]$-module structure of $T$, we compute the discriminant of the invariant lattice $T^G$ (Proposition \ref{BasicLatticeQuotient2}).
In Section \ref{Application to cohomology} we apply these results to find the discriminant of $\pi_{*}(H^{*}(X,\Z))$ (Corollary \ref{MainBasicCor}), and, using basic lattice theory, we provide an upper bound for the coefficients of surjectivity $\alpha_k(X)$ (Proposition \ref{MainBasicAlphaInequality}).
Thanks to Section \ref{BasicLatticeZModule}, we can calculate the cohomology of $G$ in Section \ref{EquivariantCohomology} (Corollary \ref{equico}); this allows us to use equivariant cohomology to compute the cohomology of non-ramified quotients.
In Section \ref{Resolution of singularities} we study resolution of singularities, which leads to the proof of the main Theorem \ref{Maincor} in Section \ref{The main theorem}.
Section \ref{application} is dedicated to some applications, namely the computation of the Beauville--Bogomolov forms of certain 
orbifolds and the description of the integral cohomology of the Hilbert scheme of two points.

~\\

\textbf{Acknowledgements.} 
I am very grateful to Samuel Boissi\`ere for helpful discussions,
to Simon Kapfer for inspiring in the problem of Section \ref{IntegralBasisHilbertScheme} and to Patrick Popescu-Pampu for very useful advices about blow-ups. I also want to thank Emilio Franco, Dimitri Markushevich and Henrique S\'a Earp for many helpful comments on some preliminary versions of this work. 
This work was supported by Fapesp grant 2014/05733-9.

\section{The $\Z[G]$-module structure of a lattice}\label{BasicLatticeZModule}
\subsection{Basic properties of lattices}
We recall the basic facts on lattices which are used in this paper (see for example \cite[Chapter 8.2.1]{Dolgachev}). 
A lattice $T$ is a free $\Z$-module endowed with a non-degenerate bilinear form. 
We denote by $\discr T$ the \emph{discriminant} of $T$,
which is the absolute value of the determinant of the bilinear form of $T$. We say that $T$ is \emph{unimodular} if its discriminant is 1.
A sublattice $N$ of $T$ is said \emph{primitive} if $T/N$ is torsion-free.

If $N$ is a sublattice of a lattice $T$ of the same rank, we have the basic formula:
\begin{equation}
[T:N]^2=\frac{\discr N}{\discr T}.
\label{BasicLatticeTheory}
\end{equation} 
If $T$ is an unimodular lattice and $N$ is a primitive sublattice, then 
\begin{equation}
\discr N=\discr N^\bot.
\label{BasicLatticeTheory2}
\end{equation}
More precisely, 
the projection
\begin{equation}
\frac{T}{N\oplus N^\bot}\rightarrow A_{N}
\label{DiscrUni}
\end{equation}
is an isomorphism, where $A_{N}:=N^\vee/N$ is the \emph{discriminant group}. 
Moreover, this provides an isometry:
\begin{equation}
\nu:A_{N}\rightarrow A_{N^\bot},
\label{BasicLatticeTheory3}
\end{equation}
between the discriminant groups.

Let $X$ be a compact connected orientable manifold of dimension $n$. We endowed $H^{*}(X,\Z)$ with the natural bilinear form $B$ defined as follows.
Let $x\in H^k(X,\Z)$ and $y\in H^q(X,\Z)$,
\begin{itemize}
\item[]
$B(x,y)=0$ if $k+q<n$,
\item[]
$B(x,y)=x\cdot y$ if $k+q=n$, where $\cdot$ refers to the cup-product. 
\end{itemize}
By Poincaré duality, it turns $H^*(X,\Z)$ into an unimodular lattice.

In our case, the lattice $T$ will be endowed with an isometry group $G$ of prime order $p$. We denote by $T^G$ the \emph{invariant lattice} of $T$.
We also denote by $(T^G)^\vee(p)$ the dual of $T^G$ with its bilinear form multiplied by $p$.
Let $G$ be an automorphism group of prime order $p$ on $X$ and $T$ an unimodular sublattice of $H^{*}(X,\Z)$; in Section \ref{pushforward}, we will see that:
$$(T^{G})^\vee(p)\simeq \pi_*(T),$$
where $\pi:X\rightarrow X/G$ is the quotient map. 
For this reason, we call $(T^{G})^\vee(p)$ the \emph{push-forward lattice} of $(T,G)$; a key step for our purpose, which will be accomplished in Section \ref{disrcpushforwardlattice}, is the calculation of its discriminant. To do so, we will use the $\Z[G]$-module structure of $T$ given by $g\cdot x=g(x)$ for all $g\in G$ and all $x\in T$. 

\subsection{The Boissi\`ere--Nieper-Wisskirchen--Sarti invariants}\label{BoissiereSartiH2}
Here we review some important notions introduced in \cite{SmithTh}.
Let $p$ be a prime number, 
$T$ a $p$-torsion-free $\Z$-module of finite rank and $G=\left\langle g\right\rangle$ a group of prime order $p$ acting linearly on $T$.
Then, $T\otimes\F$ is a $\F$-vector space equipped with a linear action of $G$. 
The minimal polynomial of $g$, as an endomorphism of $T\otimes\F$, divides $X^{p}-1=(X-1)^{p}\in\F[X]$, hence $g$ admits a Jordan normal form over $\F$. We can decompose $T\otimes\F$ as a direct sum of some $\F[G]$-modules $N_{q}$ of dimension $q$ for $1\leq q\leq p$, where $g$ acts on $N_{q}$ in a suitable basis by a matrix of the following form:
$$\begin{pmatrix}
1 & 1 & & &  \\
 & \ddots & \ddots &{\fontsize{1cm}{0.5cm}\selectfont\text{0}} & \\
 & & \ddots & \ddots & \\
 & {\fontsize{1cm}{0.5cm}\selectfont\text{0}}& & \ddots & 1\\
 & & & & 1
\end{pmatrix}$$
\begin{defi}
We define the integer $\ell_{q}(T)$ as the number of blocks of size $q$ in the Jordan decomposition of the $\F[G]$-module $T\otimes\F$, so that $T\otimes\F\simeq \oplus_{q=1}^{p} N_{q}^{\oplus \ell_{q}(T)}$. 
\end{defi}
When $p\leq 19$, we can define the $\ell_q(T)$ from another point of view.
Let $\xi_{p}$ be a primitive p-th root of unity, $K:=\mathbb{Q}(\xi_{p})$ and $\mathcal{O}_{K}:= \Z[\xi_{p}]$ the ring of algebraic integers of $K$. By a classical theorem of Masley-Montgomery \cite{Masley}, $\mathcal{O}_{K}$ is a PID if and only if $p\leq 19$. The $\Z[G]$-module structure of $\mathcal{O}_{K}$ is defined by $g\cdot x = \xi_{p}x$ for $x\in \mathcal{O}_{K}$. For any $a\in \mathcal{O}_{K}$, we denote by $(\mathcal{O}_{K},a)$ the module $\mathcal{O}_{K}+\Z$ whose $\Z[G]$-module structure is defined by $g\cdot(x,k)=(\xi_{p}x+ka,k)$.
In the proof of \cite[Proposition 5.1]{SmithTh}, using \cite[Theorem 74.3]{Reiner} of Diederichsen and Reiner, Boissi\`ere, Nieper-Wisskirchen and Sarti show the following.
\begin{prop}\label{defipro}
Let $T$ be a free $\Z$-module of finite rank and let $G=\left\langle g\right\rangle$ be a group of prime order $p\leq19$ acting on $T$. Then, we have an isomorphism of $\Z[G]$-module:
\begin{itemize}
\item[(i)]
$$T\simeq \oplus_{i=1}^{r} (\mathcal{O}_{K},a_{i})\oplus\mathcal{O}_{K}^{\oplus s}\oplus \Z^{\oplus t},$$
where  $r$, $s$, $t$ are integers and $a_{i}\notin (\xi_{p}-1)\mathcal{O}_{K}$.
Moreover, the integers $r$, $s$ and $t$ are uniquely determined by the $\Z[G]$-module structure of $T$. 
\item[(ii)]
When $p\geq3$, $(\mathcal{O}_{K},a_{i})\otimes\F=N_p$, $\mathcal{O}_{K}\otimes\F=N_{p-1}$ and $\Z\otimes\F=N_1$, so $r=\ell_p(T)$, $s=\ell_{p-1}(T)$ and $t=\ell_1(T)$.
\item[(iii)]
We have: 
$$T^G\simeq \oplus_{i=1}^{r} (\mathcal{O}_{K},a_{i})^G\oplus\Z^{\oplus t},$$
with $\rk(\mathcal{O}_{K},a_{i})^G=1$.
\end{itemize}
\end{prop}
\begin{rmk}
In the particular case $p=2$, we have $\ell_{1}(T)=t+s$ and $\ell_{2}(T)=r$.
\end{rmk}
Hence, to distinguish the integers $t$ and $s$ in the case $p=2$, we choose the following notation.
\begin{defi}
When $p=2$, we define $\ell_+(T):=t$, $\ell_-(T):=s$.
\end{defi}
\begin{rmk}
When $p=2$, it follows: $\ell_1(T)=\ell_+(T)+\ell_-(T)$.
\end{rmk}
Moreover, to avoid distinguishing the cases $p=2$ and $p\geq3$ all over, we adopt the following notation.
\begin{nota}
When $p\geq 3$, we also denote $\ell_+(T):=\ell_1(T)=t$ and $\ell_-(T):=\ell_{p-1}(T)=s$.
\end{nota}
We can reformulate Proposition \ref{defipro} in those terms:
\begin{cor}\label{BasicLatticeQuotient}
Let $T$ be a free $\Z$-module and let $G=\left\langle g\right\rangle$ be a group of prime order $p\leq19$ acting on $T$.
We have:
\begin{itemize}
\item[(i)]
$\rk T=p\ell_{p}(T)+(p-1)\ell_{-}(T)+\ell_{+}(T)$,
\item[(ii)]
$\rk T^{G}=\ell_{p}(T)+\ell_{+}(T)$.
\end{itemize}
\end{cor}
As in \cite{SmithTh}, in the case of cohomology, we choose the following notation for our invariants:
\begin{nota}
For all $0\leq k\leq \dim X$, we denote:
$$\ell_+^k(X)=\ell_+(H^k_f(X,\Z)),\ \ell_-^k(X)=\ell_-(H^k_f(X,\Z))\ \text{and for all}\ q\in\left\{1,...,p\right\},\ \ell_q^k(X)=\ell_q(H^k_f(X,\Z)).$$
\end{nota}
\subsection{Discriminant of the push-forward lattice}\label{disrcpushforwardlattice}
\begin{prop}\label{MainBasicLatticeQuotient}
Let $T$ be an unimodular lattice and let $G=\left\langle g\right\rangle$ be a group of prime order $p$ acting on $T$. Then,
$$\discr (T^G)^\vee(p)=p^{\rk T^G-\ell_{p}(T)}.$$
If moreover $p\leq 19$, we have:
$$\discr (T^G)^\vee(p)=p^{\ell_{+}(T)}.$$
\end{prop}
This section is dedicated to proving this proposition. We need several lemmas.
\begin{lemme}\label{BasicLatticeQuotientLemma}
Let $T$ be a lattice and $G=\left\langle g\right\rangle$ be a group of prime order $p$ acting on $T$. Let $\sigma=\id+g+...+g^{p-1}$
then: $$\ker\sigma=(T^{G})^\bot.$$
\end{lemme}
\begin{proof}
We have $\ker\sigma\subset(T^{G})^\bot$. Indeed, let $x\in \ker \sigma$ and $y\in T^{G}$. For all $0\leq k\leq p-1$, we have:
$g^{k}(x)\cdot y=x\cdot y$. Hence, $0=\sigma(x)\cdot y=px\cdot y$. So, $x\cdot y=0$.
Conversely, let $x\in (T^{G})^\bot$, let $y=x+g(x)+...+g^{p-1}(x)$. 
Then, we have $y\in (T^{G})^\bot\cap T^{G}$.
Since the bilinear form on $T$ is non-degenerate, we have $y=0$.
\end{proof}
\begin{lemme}\label{latima}

Let $T$ be an unimodular lattice and $G=\left\langle g\right\rangle$ be a group of prime order $p$ acting on $T$. 
$$(T^{G})^{\vee}=\left\{\left.\frac{y+g(y)+...+g^{p-1}(y)}{p}\right|\ y\in T \right\}.$$
\end{lemme}
\begin{proof}
For all $x\in T^{G}$ and all $y\in T$,
$$\frac{y+g(y)+...+g^{p-1}(y)}{p}\cdot x= y\cdot x.$$
Hence: 
$$(T^{G})^{\vee}\supset\left\{\left.\frac{y+g(y)+...+g^{p-1}(y)}{p}\right|\ y\in T \right\}.$$
Now, we prove the converse inclusion.
Let $x$ be a primitive element of $T^{G}$ and $q\in\mathbb{N}^{*}$ such that $\frac{x}{q}\in (T^{G})^{\vee}$. Then, it provides $\overline{\frac{x}{q}}\in A_{T^{G}}$.
By (\ref{DiscrUni}), there is $z\in T$ and $y\in (T^G)^\bot$ such that $z=\frac{x+y}{q}$. 
Then, $z+g(z)+...+g^{p-1}(z)=\frac{p}{q}x+\frac{y+g(y)+...+g^{p-1}(y)}{q}$.
But by Lemma \ref{BasicLatticeQuotientLemma},
$y+g(y)+...+g^{p-1}(y)=0$.
Hence, $z+g(z)+...+g^{p-1}(z)=\frac{p}{q}x$. Since $x$ is primitive in $T^{G}$, $q$ divides $p$. Hence, $q=1$ or $q=p$.
If $q=p$, we get $z+g(z)+...+g^{p-1}(z)=x$ so that $\frac{x}{q}=\frac{z+g(z)+...+g^{p-1}(z)}{p}$. If $q=1$, we can write $x=\frac{x+...+x}{p}=\frac{x+g(x)+...+g^{p-1}(x)}{p}$ since $x\in T^{G}$. 
\end{proof}
From a basic calculation (see proof of \cite[Lemma 3.1]{SmithTh}), we obtain:
\begin{lemme}\label{adequi}
Let $T$ be a free $\Z$-module and $G=\left\langle g\right\rangle$ be a group of prime order $p$ acting on $T$. We denote $\tau=g-\id$ and $\sigma=\id+g+g^2+...+g^{p-1}$. Moreover, we denote by $\overline{\tau}, \overline{\sigma}\in \F[G]$ respectively the reduction of $\tau$ and $\sigma$ modulo $p$.
Let $(v_{1},...,v_{q})$ be the canonical basis of the $\F[G]$-module $N_{q}$ defined in Section \ref{BoissiereSartiH2}; that is $g (v_{1})=v_{1}$ and $g (v_{i})=v_{i}+v_{i-1}$ for all $i\geq 2$.
Then the following hold:
\begin{itemize}
\item[(i)]
$\ker(\overline{\tau})=\left\langle v_{1}\right\rangle$ and $\Image(\overline{\tau})=\left\langle v_{1},...,v_{q-1}\right\rangle$,
for all $q\leq p$.
\item[(ii)]
$\ker(\overline{\sigma})=N_{q}$ and $\Image(\overline{\sigma})=0$, for all $q<p$.
Moreover, $\ker(\overline{\sigma})=\left\langle v_{1},...,v_{q-1}\right\rangle$ and $\Image(\overline{\sigma})=\left\langle v_{1}\right\rangle$, if $q=p$.
\end{itemize}
\end{lemme}
\begin{lemme}\label{?}
Let $T$ be a free $\Z$-module and $G=\left\langle g\right\rangle$ be a group of prime order $p$ acting on $T$.
Following notation of Section \ref{BoissiereSartiH2}, we have:
$$T\otimes\F\simeq \oplus_{q=1}^{p} N_{q}^{\oplus \ell_{q}(T)}.$$
For all $z\in T$, we denote now by $\overline{z}\in T\otimes\F$ its reduction modulo $p$.

Let $x\in T^G$, then there exists $y\in T$ such that $x=y+g(y)+...+g^{p-1}(y)$ if and only if $\overline{x}\in N_{p}^{\oplus \ell_{p}(T)}$.
\end{lemme}
\begin{proof}
Assume that $x=y+g(y)+...+g^{p-1}(y)$, with $y\in T$.
If $\overline{x}=0$, then $\overline{x}\in N_{p}^{\oplus \ell_{p}(T)}$.
Now we assume that $\overline{x}\neq 0$.
Then, $\overline{y}\notin \ker \overline{\sigma}$, so by Lemma \ref{adequi}, (ii), $\overline{y}\in N_{p}^{\oplus \ell_{p}(T)}$.
Hence, $\overline{x}\in N_{p}^{\oplus \ell_{p}(T)}$.

Conversely, assume that $\overline{x}\in N_{p}^{\oplus \ell_{p}(T)}$, we can write $\overline{x}=\sum_{i} a_{i} v_{1,i}$, where $v_{1,i}$ are invariant elements of the direct summands of $N_{p}^{\oplus \ell_{p}(T)}$ (see Lemma \ref{adequi}, (i)).
But, we have $v_{1,i}=v_{p,i}+g(v_{p,i})+...+g^{p-1}(v_{p,i})$ by Lemma \ref{adequi}, (ii).
The result follows.
\end{proof}
\begin{prop}\label{BasicLatticeQuotient3}
Let $T$ be an unimodular lattice and let $G=\left\langle g\right\rangle$ be a group of prime order $p$ acting on $T$. Then:
$$\frac{T}{T^G\oplus (T^G)^\bot}=(\Z/p\Z)^{\ell_p(T)}.$$
\end{prop}
\begin{proof}
Let $T\otimes\F\simeq \oplus_{q=1}^{p} N_{q}^{\oplus \ell_{q}(T)}$ be the isomorphism of $\F[G]$-module provided in Section \ref{BoissiereSartiH2}.
By Lemma \ref{latima} and \ref{?}, we can define a surjective morphism:
$$\xymatrix@R=0pt{(T^G)^\vee\ar[r] & (N_{p}^G)^{\oplus \ell_{p}(T)}\\
\frac{x}{p}\ar[r]& \overline{x},}$$
where $\overline{x}$ denotes the reduction modulo $p$ of $x$.
Moreover, $\frac{x}{p}\in T^G\Leftrightarrow p|x\Leftrightarrow\overline{x}=0$.
Hence, the previous morphism provides an isomorphism
$$A_{T^G}\simeq (N_{p}^G)^{\oplus \ell_{p}(T)}.$$
Since $N_{p}^G\simeq(\Z/p\Z)$, the result follows by (\ref{DiscrUni}).
\end{proof}
Now,
by Proposition \ref{BasicLatticeQuotient3} and (\ref{BasicLatticeTheory}), we have 
\begin{equation}
\discr \left(T^{G}\oplus (T^{G})^\bot\right)=p^{2\ell_p(T)}.
\label{discrofTG}
\end{equation}
So, from (\ref{BasicLatticeTheory2}) and (\ref{discrofTG}), we obtain the discriminant of $T^G$:
\begin{prop}\label{BasicLatticeQuotient2}
Let $T$ be an unimodular lattice and let $G=\left\langle g\right\rangle$ be a group of prime order $p$ acting on $T$. Then:
$$\discr T^{G}=p^{\ell_p(T)}.$$
\end{prop}
Now, we conclude the proof of Proposition \ref{MainBasicLatticeQuotient}.
Since $T$ is unimodular, we know by Proposition \ref{BasicLatticeQuotient3} that the discriminant group $A_{T^G}:=\frac{(T^G)^\vee}{T^G}$ of $T^G$ is isomorphic to $\frac{T}{T^G\oplus (T^G)^\bot}=(\Z/p\Z)^{\ell_p(T)}$.
Hence, by Proposition \ref{BasicLatticeQuotient2} and (\ref{BasicLatticeTheory}):
$$\discr  (T^G)^\vee=p^{-\ell_p(T)}.$$
It follows: 
$$\discr  (T^G)^\vee(p)=p^{\rk T^G-\ell_p(T)}.$$
So, by Corollary \ref{BasicLatticeQuotient}, in the case $p\leq 19$, we obtain:
$$\discr  (T^G)^\vee(p)=p^{\ell_+(T)}.$$
\section{The $\Z[G]$-module structure of 
the cohomology lattice}\label{Application to cohomology}
In all Section \ref{Application to cohomology}, $X$ be a compact connected orientable manifold of dimension $n$ and $G=\left\langle g\right\rangle$ be an automorphism group of prime order $p$. We denote by $\pi:X\rightarrow X/G$ the quotient map.
\subsection{Coefficients of surjectivity}\label{DefinitionIntro}
The following proposition is straight-forward from \cite[Theorem 5.4 and Corollary 5.8]{Transfers} (it is a generalization from homology to cohomology of original results of Smith \cite{Smith})
\begin{prop}\label{SmithProp}
We can define a push-forward map $\pi_{*}:H^{*}(X,\Z)\rightarrow H^{*}(X/G,\Z)$ such that:
$$\pi_{*}\circ\pi^{*}=p\id_{H^{*}(X/G,\Z)}, \ \ \ \ \ \ \pi^{*}\circ\pi_{*}=\sum_{g\in G}{g^{*}}.$$
\end{prop}
From this proposition, let us examine the torsion of $H^{*}(X/G,\Z)$.
Notice that $H^*(X,\Z)$ can be decomposed into three parts: the torsion-free part, the $p$-torsion part and the other torsion part. 
\begin{rmk}
Denote by $H^*_{o-t}(X,\Z)$ the torsion part of $H^*(X,\Z)$ without the $p$-torsion part. By using Proposition \ref{SmithProp}
$$H^*_{o-t}(X,\Z)^G\simeq H^*_{o-t}(X/G,\Z).$$
\end{rmk}
Hence, the torsion which is not $p$-torsion does not provide any difficulties. On the other hand, the technique we propose does not say anything about 
the $p$-torsion part of $H^*(X/G,\Z)$.
So, in this paper, we will study only the torsion-free part of $H^*(X/G,\Z)$.
\begin{nota}
Let $T$ be a $\Z$-module. We denote: 
\begin{itemize}
\item
$T_f:=T/\tors$
\item
$T_{p-f}:=T/(p-\tors)$
\end{itemize}
Let $Y$ be a topological space and $p$ a prime number. 
Set:
\begin{itemize}
\item
 $H^*_f(Y,\Z):=H^*(Y,\Z)/\tors$, 
 \item
 $H^*_{p-f}(Y,\Z):=H^*(Y,\Z)/(p-\tors)$.
 \end{itemize}
\end{nota}
We want to calculate the 
integer $\alpha_{k}(X)$ provided by the  
exact sequence (\ref{MainExactSequence}):
$$\xymatrix{ 0\ar[r]&\pi_{*}(H^{k}(X,\Z))_f\ar[r] & H^{k}_f(X/G,\Z)\ar[r] & (\Z/p\Z)^{\alpha_{k}(X)}\ar[r]& 0.}$$
\begin{defi}\label{CoefficientOfSurjectivity}
The integer $\alpha_{k}(X)$ is called the $k$-th coefficient of surjectivity of $(X,G)$.
\end{defi}
Hence, with our notation, $\alpha_k(X)$ is expressed as follows:
$$\frac{H^k_{p-f}(X/G,\Z)}{\pi_*(H^k(X,\Z))_{p-f}}=\frac{H^k_{f}(X/G,\Z)}{\pi_*(H^k(X,\Z))_{f}}=(\Z/p\Z)^{\alpha_k(X)}.$$
Then, from (\ref{BasicLatticeTheory}), we obtain the following formula:
\begin{equation}
\sum_{k=0}^{n} \alpha_{k}(X)=\frac{\log_p\discr \pi_{*}(H^{*}(X,\Z))_f-\log_p\discr H^{*}_f(X/G,\Z)}{2}.
\label{BasicAlphaEquality0}
\end{equation}
This formula is available under many different versions considering natural sublattices of $H^{*}(X,\Z)$ as $H^{k}(X,\Z)\oplus H^{n-k}(X,\Z)$:
\small
\begin{equation}
\alpha_{k}(X)+\alpha_{n-k}(X)=\frac{\log_p\discr \left(\pi_{*}\left(H^{k}(X,\Z)\oplus H^{n-k}(X,\Z)\right)_f\right)-\log_p\discr \left(H^{k}_f(X/G,\Z)\oplus H^{n-k}_f(X/G,\Z)\right)}{2}.
\label{BasicAlphaEquality02}
\end{equation}
\normalsize
In the rest of this section, we will see how this formula leads to an upper bound on the coefficients of surjectivity.
\subsection{The push-forward lattice}\label{pushforward}
Now in order to apply Proposition \ref{MainBasicLatticeQuotient}, in this section we prove the following proposition.
\begin{prop}\label{MainDiscrCoho}
Let $T$ be a unimodular sublattice of $H^*(X,\Z)$ stable under the action of $G$.
Then, we have an isometry: $$\pi_*(T)_f\simeq (T^G)^\vee(p).$$
\end{prop}
Before proving the proposition, the following lemma allows us to understand the cup-product form of the lattice $\pi_*(H^*(X,\Z)^G)$.
\begin{lemme}\label{commut}
\begin{itemize}
\item[(i)]
Let $0\leq k \leq n$, $q$ an integer such that $kq\leq n$ and $x\in H^{k}(X,\Z)^{G}$.
Then: $$\pi_{*}(x)^{q}=p^{q-1}\pi_{*}(x^{q})+(p-\tors).$$ If moreover $H^{kq}(X,\Z)$ is $p$-torsion-free, then the property that $\pi_{*}(x)$ is divisible by $p$ implies that $\pi_{*}(x^{q})$ is divisible by $p$.
\item[(ii)]
Let $(x_{i})_{1\leq i \leq q}$ be elements of $H^{k_i}(X,\Z)^{G}$ with $(k_i)_{1\leq i \leq q}$ integers such that $k_1+...+k_q\leq n$, then:
$$\pi_{*}(x_{1})\cdot...\cdot \pi_{*}(x_{q})=p^{q-1}\pi_{*}(x_{1}\cdot...\cdot x_{q})+(p-\tors).$$
\end{itemize}
\end{lemme}
\begin{proof}
\begin{itemize}
\item[(i)]
By Proposition \ref{SmithProp}:
$$\pi^{*}(\pi_{*}(x)^{q})= p^{q}x^{q}=\pi^{*}(p^{q-1}\pi_{*}(x^{q})).$$
The map $\pi^{*}$ is injective on the $p$-torsion-free part, which implies the claimed equality.
If moreover $\pi_{*}(x)$ is divisible by $p$, we can write 
$\pi_{*}(x)=py$ with $y\in H^{kq}(X/G,\Z)$.
This gives:
\begin{equation}
p^{q}y^{q}=p^{q-1}\pi_{*}(x^{q})+(p-\tors).
\label{sameproofcommut}
\end{equation}
We cannot divide by $p^{q-1}$ because of the possible torsion of $H^{kq}(X/G,\Z)$. We will use the fact that $H^{kq}(X,\Z)$ is $p$-torsion-free to get around the problem.
Applying $\pi^{*}$ to this equality, we obtain:
$$p^{q}\pi^{*}(y^{q})=p^{q}x^{q}+\pi^{*}(p-\tors).$$
Since $H^{kq}(X,\Z)$ is $p$-torsion-free, $\pi^{*}(p-\tors)=0$, and we have
$$\pi^{*}(y^{q})=x^{q}.$$
Pushing down by $\pi_{*}$, we obtain: $$py^{q}=\pi_{*}(x^{q}).$$
\item[(ii)]
Same proof as (\ref{sameproofcommut}).
\end{itemize}
\end{proof}
\begin{rmk}
Very often in this paper, we will be working in $H_f^*(X/G,\Z)$, in this case, 
the torsion part in the previous equations can be omitted. 
In particular, when $k_1+...+k_q=n$, identifying $H^n(X,\Z)$ and $H_f^n(X/G,\Z)$ with $\Z$, we get:
$$\pi_{*}(x_{1})\cdot...\cdot \pi_{*}(x_{q})=p^{q-1}\pi_{*}(x_{1}\cdot...\cdot x_{q})=p^{q-1}x_{1}\cdot...\cdot x_{q}.$$
\end{rmk}
Now, we prove Proposition \ref{MainDiscrCoho}.
By Lemma \ref{latima}, we have:
$$\pi_{*}(T)_f\supset\pi_{*}((T^{G})^\vee).$$
Indeed, let $z\in (T^{G})^\vee$, by Lemma \ref{latima} there exist $y\in T$ such that $z=\frac{y+g(y)+...+g^{p-1}(y)}{p}$.
So, $\pi_*(y)=\pi_*(z)$.
Now let us prove the other inclusion:
$$\pi_{*}(T)_f\subset\pi_{*}((T^{G})^\vee).$$
Let $z\in T$, by Proposition \ref{BasicLatticeQuotient3}, we can write $z=\frac{x+y}{p}$ with $x\in T^G$ and $y\in (T^G)^\bot$. 
Hence, by Lemma \ref{BasicLatticeQuotientLemma}, we obtain $x=z+g(z)+...+g^{p-1}(z)$.
Hence, $\pi_*(z)=\pi_*(\frac{x}{p})\in \pi_{*}((T^{G})^\vee)$.
It follows $$\pi_{*}(T)_f=\pi_{*}((T^{G})^\vee).$$
Then, by Lemma \ref{commut} (ii), we have an isometry:
$$\pi_{*}(T)_f\simeq(T^{G})^\vee(p).$$
That proves Proposition \ref{MainDiscrCoho}.

Moreover, combining Proposition \ref{MainBasicLatticeQuotient} and \ref{MainDiscrCoho}, we obtain the following corollary.
\begin{cor}\label{MainBasicCor}
Let $T$ be a unimodular sublattice of $H^*(X,\Z)$ stable under the action of $G$;
then
$$\discr \pi_*(T)_f=p^{\rk T^G-\ell_{p}(T)}.$$
If moreover $p\leq 19$, we have:
$$\discr \pi_*(T)_f=p^{\ell_{+}(T)}.$$
\end{cor}
\subsection{Upper bound for the coefficients of surjectivity}
If we apply Corollary \ref{MainBasicCor}, to the lattice $H^{k}_f(X,\Z)\oplus H^{n-k}_f(X,\Z)$, we obtain:
\begin{equation}
\discr \pi_*(H^{k}(X,\Z)\oplus H^{n-k}(X,\Z))_f=p^{\rk \left(H^{k}_f(X,\Z)^G\oplus H^{n-k}_f(X,\Z)^G\right)-\ell_p^k(X)-\ell_p^{n-k}(X)}.
\label{BasicAlphaEquality}
\end{equation}
Using (\ref{BasicAlphaEquality02}), it follows that
\begin{equation}
\alpha_k(X)+\alpha_{n-k}(X)\leq \frac{\rk \left(H^{k}_f(X,\Z)^G\oplus H^{n-k}_f(X,\Z)^G\right)-\ell_p^k(X)-\ell_p^{n-k}(X)}{2}.
\label{BasicAlphaInequality}
\end{equation}
\begin{prop}\label{BasicAlphaInequality2}
For all $0\leq k \leq n$:
\begin{itemize}
\item[(i)]
$\rk H_f^k(X,\Z)^G=\rk H_f^{n-k}(X,\Z)^G$,
\item[(ii)]
$\ell_*^k(X)=\ell_*^{n-k}(X),$
where $*$ means $+$, $-$ or $q\in\left\{1,...,p\right\}$.
\end{itemize}
\end{prop}
\begin{proof}
By Poincar\'e duality, we have:
\begin{equation}
\xymatrix@R=0pt{H^k_f(X,\Z)\ar[r] & H^{n-k}_f(X,\Z)^\vee\\
x\ar[r]& f_x:=(y\mapsto x\cdot y),}
\label{Poincare}
\end{equation}
which is an isomorphism.
Moreover, we can endow $H^{n-k}_f(X,\Z)^\vee$ with a structure of $\Z[G]$-module by setting $g\cdot f=f(g^{-1}\cdot\ )$ for all $f\in H^{n-k}_f(X,\Z)^\vee$ and $g\in G$.
With this structure, the Poincar\'e isomorphism becomes an isomorphism of $\Z[G]$-modules.
So, it shows that:
$$\ell_*^k(X)=\ell_*(H^k_f(X,\Z))=\ell_*(H^{n-k}_f(X,\Z)^\vee).$$ 
Now, it remains to show that 
$$\ell_*(H^{n-k}_f(X,\Z)^\vee)=\ell_*(H^{n-k}_f(X,\Z))=\ell_*^{n-k}(X).$$
At this stage of the proof, the case $p=2$ differs from the case $p\geq3$.

Let us first assume that $p\geq3$.
We have to prove that, for $1\leq q\leq p$, the $\F[G]$-modules $N_q$ and $N_q^\vee$ are isomorphic (where $N_q$ was defined in Section \ref{BoissiereSartiH2}).
The matrix of the action of $g$ on $N_q^\vee$ is given by the inverse of the transpose matrix of the Jordan block of size $q$:
$$\left(\begin{pmatrix}
1 & 1 & & &  \\
 & \ddots & \ddots &{\fontsize{1cm}{0.5cm}\selectfont\text{0}} & \\
 & & \ddots & \ddots & \\
 & {\fontsize{1cm}{0.5cm}\selectfont\text{0}}& & \ddots & 1\\
 & & & & 1
\end{pmatrix}^{-1}\right)^t.$$
Denote by $J_q$ the matrix of the Jordan block of size $q$.
The minimal polynomial of $J_q$ is $(X-1)^q$.
However, $((J_q^{-1})^t)^q=((J_q^q)^{-1})^t=I_q$,
with $I_q$ the identity matrix of size $q$.
So, the minimal polynomial of $(J_q^{-1})^t$ divides $(X-1)^q$.
Let $0\leq d\leq q$ minimal integer such that $((J_q^{-1})^t)^d=I_q$.
However, this implies $J_q^d=I_q$.
So, $d=q$.
It follows that the Jordan form of $(J_q^{-1})^t$ is given by $J_q$.
Hence, $N_q\simeq N_q^\vee$ as $\F[G]$-module. This proves the result for $p\geq3$.

Now assume that $p=2$.
It is clear that the dual of $\Z$ endowed with the trivial $\Z[G]$-module structure is again $\Z$ endowed with the trivial $\Z[G]$-module structure.
The $\Z[G]$-module $\mathcal{O}_{K}$ defined in Section \ref{BoissiereSartiH2}, is only $\Z$ endowed with the $\Z[G]$-module structure given by $g\cdot t=-t$ for all $t\in \Z$.
Hence, we also have an isomorphism of $\Z[G]$-module between $\mathcal{O}_{K}$ and $\mathcal{O}_{K}^\vee$ endowed with the $\Z[G]$-module structure given by 
$g\cdot f=f(g^{-1}\cdot\ )$ for all $f\in\mathcal{O}_{K}^\vee$.
Let us consider now $(\mathcal{O}_{K},a)^\vee$ endowed with the $\Z[G]$-module structure given by 
$g\cdot f=f(g^{-1}\cdot\ )$ for all $f\in (\mathcal{O}_{K},a)^\vee$.
We can apply Proposition \ref{defipro}, (i) to $(\mathcal{O}_{K},a)^\vee$. 
If $\ell_{-}((\mathcal{O}_{K},a)^\vee)$ and $\ell_{+}((\mathcal{O}_{K},a)^\vee)$ are not both zero, we obtain a contradiction by considering the double dual which is the dual of $(\mathcal{O}_{K},a)^\vee$ and is canonically isomorphic to $(\mathcal{O}_{K},a)$.
Hence, there exists $b\in \mathcal{O}_{K}$ such that $(\mathcal{O}_{K},a)^\vee\simeq (\mathcal{O}_{K},b)$ as $\Z[G]$-module. This concludes the proof of (ii).

It remains to prove (i).
By Poincar\'e duality isomorphism (\ref{Poincare}), we only have
to show that $\rk H^{n-k}_f(X,\Z)^G=\rk (H^{n-k}_f(X,\Z)^\vee)^G$,
where the action of $G$ on $H^{n-k}_f(X,\Z)^\vee$ is given by
$g\cdot f=f(g^{-1}\cdot\ )$ for all $f\in H^{n-k}_f(X,\Z)^\vee$.
This is equivalent to show that
$\dim H^{n-k}(X,\Q)^G =\dim (H^{n-k}(X,\Q)^\vee)^G$,
which is readily verified considering a basis $\mathcal{B}=\left\{v_1,...,v_d,v_{d+1},...,v_{m}\right\}$ of $H^{n-k}(X,\Q)$ 
where $\left\{v_1,...,v_d\right\}$ is a basis of $H^{n-k}(X,\Q)^G$
and taking the dual basis of $\mathcal{B}$.
\end{proof}

Then, it follows from (\ref{BasicAlphaInequality}) (and from Corollary \ref{BasicLatticeQuotient} in the case $p\leq 19$):
\begin{prop}\label{MainBasicAlphaInequality}
For all $0\leq k \leq n$:
$$\alpha_k(X)+\alpha_{n-k}(X)\leq \rk H^{k}_f(X,\Z)^G-\ell_p^k(X);$$
if moreover $p\leq 19$, then
$$\alpha_k(X)+\alpha_{n-k}(X)\leq \ell_+^k(X).$$
Furthermore, these inequalities become equalities when $X/G$ is smooth.
\end{prop}
\begin{proof}
The first inequality follows from (\ref{BasicAlphaInequality}) and Proposition \ref{BasicAlphaInequality2}. Then, we deduce the second one from Corollary \ref{BasicLatticeQuotient}.
When $X/G$ is smooth, the lattice $H^k_f(X/G,\Z)\oplus H^{n-k}_f(X/G,\Z)$ is unimodular.
Therefore, (\ref{BasicAlphaEquality02}), (\ref{BasicAlphaEquality}) and Proposition \ref{BasicAlphaInequality2} provide the equality.
\end{proof}
\begin{rmk}\label{MainBasicAlphaInequalityRemark}
When $n=2m$ is even and $k=m$, we calculated the discriminant of two copies of the lattice $\pi_*(H^{m}_f(X,\Z))$ and the inequalities of Proposition \ref{MainBasicAlphaInequality} are: 
$$\alpha_{m}(X)\leq \frac{\rk H^{m}_f(X,\Z)^G-\ell_p^{m}(X)}{2};$$
and if $p\leq19$:
$$\alpha_{m}(X)\leq \frac{\ell_+^{m}(X)}{2}.$$
\end{rmk}
\subsection{The $H^k$-normality}
A particular case of interest is when a coefficient of surjectivity $\alpha_k(X)$ vanishes, for some $k$. In such a case, the push-forward map 
 $\pi_{*}: H^{k}_{p-f}(X,\Z)\rightarrow H^{k}_{p-f}(X/G,\Z)$ is surjective.
\begin{defi}
Let $0\leq k\leq n$. 
When $\alpha_k(X)=0$, we say that $(X,G)$ is \emph{$H^{k}$-normal}.
\end{defi}
\begin{prop}\label{lalalalala}
Let $0\leq k\leq n$. The following statements are equivalent:
\begin{itemize}
\item[(i)]
$(X,G)$ is $H^k$-normal.
\item[(ii)]
For $x\in H^{k}_{f}(X,\Z)^{G}$, $\pi_{*}(x)$ is divisible by $p$ if and only if there exists $y\in H^{k}_{f}(X,\Z)$ such that 
$x=y+g(y)+...+g^{p-1}(y)$.
\end{itemize}
\end{prop}
\begin{proof}
We assume (ii). Let $z\in H^{k}_{f}(X/G,\Z)$, by (\ref{MainExactSequence}), we can find $x\in H^{k}_{f}(X,\Z)$ such that $z=\frac{\pi_*(x)}{p}$.
So by (ii), there exists $y\in H^{k}_{f}(X,\Z)$ such that $x=y+g(y)+...+(g)^{p-1}(y)$. Hence $z=\pi_*(y)$. 

Now, assume (i).
Let $x\in H^{k}_{f}(X,\Z)^{G}$ such that $p$ divides $\pi_{*}(x)$. Since $\pi_{*}: H^{k}_{f}(X,\Z)\rightarrow H^{k}_{f}(X/G,\Z)$ is surjective, there is $y\in H^{k}_{f}(X,\Z)$ such that $p\pi_{*}(y)=\pi_{*}(x)$. We apply $\pi^{*}$ and Proposition \ref{SmithProp} to get $p(y+g(y)+...+g^{p-1}(y))=px$.
\end{proof}
We provide two important properties of $H^k$-normality.
First, in certain circumstances, it is possible to deduce the $H^{k}$-normality from $H^{kt}$-normality. 
\begin{prop}\label{Hkt}
We assume $p\leq 19$ and $H^{*}(X,\Z)$ $p$-torsion-free (so $H^{*}(X,\Z)\otimes\F=H^{*}(X,\F)$).
Let $0\leq k \leq n$, $t$ an integer such that $kt\leq n$. Assume that $(X,G)$ is $H^{kt}$-normal. For all $x\in H^*(X,\Z)$, we denote by $\overline{x}\in H^{*}(X,\F)$ its reduction modulo $p$.
If 
\begin{align*}
\mathscr{S}:\ &\Sym^t H^{k}(X,\F)\rightarrow H^{kt}(X,\F)\\
&\overline{x_{1}}\otimes...\otimes \overline{x_{t}}\mapsto \overline{x_{1}\cdot...\cdot x_{t}}
\end{align*}
is injective and $\mathscr{S}(\Sym^t H^{k}(X,\F))$ admits a complementary vector space, stable by the action of $G$, then $(X,G)$ is $H^{k}$-normal.
\end{prop}
\begin{proof}
We use the following notation for the Jordan decompositions of $H^{k}(X,\F)$ and of $H^{kt}(X,\F)$:
\begin{equation}
H^{k}(X,\F)=\sum_{q=1}^{p} N_{q}^{\oplus l_{q}^{k}(X)}:=\sum_{q=1}^{p} \mathcal{N}_{q}^k\ \text{and}\ H^{kt}(X,\F)=\sum_{q=1}^{p} N_{q}^{\oplus l_{q}^{kt}(X)}:=\sum_{q=1}^{p}\mathcal{N}_{q}^{kt}.
\label{normality1}
\end{equation}
Moreover, since $p\leq 19$:
$$H^{k}(X,\F)=\mathcal{N}_{1}^k\oplus \mathcal{N}_{p-1}^k\oplus\mathcal{N}_{p}^k\ \text{and}\ H^{kt}(X,\F)=\mathcal{N}_{1}^{kt}\oplus \mathcal{N}_{p-1}^{kt}\oplus\mathcal{N}_{p}^{kt}.$$
Hence, by Proposition \ref{defipro}:
\begin{equation}
H^{k}(X,\F)^G=\mathcal{N}_{1}^k\oplus(\mathcal{N}_{p}^k)^G.
\label{normality11}
\end{equation}

Let $x\in H^{k}(X,\Z)^{G}$, we are going to use the characterization of Proposition \ref{lalalalala}.
We assume that there is no $y\in H^{k}(X,\Z)$ such that 
$x=y+g(y)+...+g^{p-1}(y)$ and we show that $\pi_{*}(x)$ is not divisible by $p$.

By Lemma \ref{?}, $\overline{x}\notin \mathcal{N}_{p}^k$.
Since $\mathscr{S}$ is injective and $N_{1}^{\otimes t}=N_{1}$, by (\ref{normality11}), we have $\overline{x^{t}}\notin \mathcal{N}_{p}^{kt}$.
Again by Lemma \ref{?}, there is no $z\in H^{kt}(X,\Z)$ such that 
$x^{t}=z+g(z)+...+g^{p-1}(z)$.
Since  $(X,G)$ is $H^{kt}$-normal, $\pi_{*}(x^{t})$ is not divisible by $p$.
Now, since $H^{kt}(X,\Z)$ is $p$-torsion-free, by Lemma \ref{commut} (i), $\pi_{*}(x)$ is not divisible by $p$.
\end{proof}
We will apply this proposition to an example in Section \ref{K3typeApplication}. 
Now, we introduce the notion of the \emph{pullback} of a pair $(X,G)$.

\begin{defi}\label{pullback}

Let $s:\widetilde{X}\rightarrow X$ be a surjective morphism such that $\widetilde{X}$ is a compact connected orientable manifold of dimension $n$ and suppose that there is an automorphism $\widetilde{g}$ of order $p$ of $\widetilde{X}$ which verifies $s\circ \widetilde{g}=g \circ s$. Moreover, we assume that the degree of $s$ is not divisible by $p$. We denote $\left\langle \widetilde{g}\right\rangle$ by $\widetilde{G}$.

The triple $(\widetilde{X},\widetilde{G},s)$ will be called a \emph{pullback} of $(X,G)$.
\end{defi}
\begin{lemme}\label{IlLeFautBien}
Let $(\widetilde{X},\widetilde{G},s)$ be a pullback of $(X,G)$. We denote by $\widetilde{\pi}: \widetilde{X}\rightarrow\widetilde{X}/\widetilde{G}$ the quotient map and by $r:\widetilde{X}/\widetilde{G}\rightarrow X/G$ the induced map. Let $0\leq k \leq n$ and $x\in H^{k}(X,\Z)^{G}$. 
Then:
$$\widetilde{\pi}_{*}(s^{*}(x))=r^{*}(\pi_{*}(x))+(p-\tors).$$
If moreover $H^{k}(\widetilde{X},\Z)$ is $p$-torsion-free, then the property that $r^{*}(\pi_{*}(x))$ is divisible by $p$ implies that $\widetilde{\pi}_{*}(s^{*}(x))$ is divisible by $p$.
\end{lemme}
\begin{proof}
The proof follows the same idea as the proof of Lemma \ref{commut}.

We have a Cartesian diagram

$$\xymatrix{
 \widetilde{X}\ar[d]^{\widetilde{\pi}} \ar[r]^{s}& X\ar[d]^{\pi} \\
    \widetilde{X}/\widetilde{G}\ar[r]^{r}  & X/G.
   }$$
It induces a commutative diagram in cohomology:

$$\xymatrix{
H^{k}(X/G,\Z)\ar@/_/[r]_{\pi^{*}} \ar[d]_{r^{*}}& H^{k}(X,\Z)\ar@/_/[l]_{\pi_{*}}\ar[d]^{s^{*}}\\
H^{k}(\widetilde{X}/\widetilde{G},\Z)\ar@/^/[r]^{\widetilde{\pi}^{*}}& \ar@/^/[l]^{\widetilde{\pi}_{*}} H^{k}(\widetilde{X},\Z).
}$$
It follows from Proposition \ref{SmithProp} that
$$\widetilde{\pi}^{*}(r^{*}(\pi_{*}(x)))=s^{*}(\pi^{*}(\pi_{*}(x)))=p\cdot s^{*}(x)= \widetilde{\pi}^{*}(\widetilde{\pi}_{*}(s^{*}(x))).$$
The map $\widetilde{\pi}^{*}$ is injective on the torsion-free part, so we get the equality.
If moreover $r^{*}(\pi_{*}(x))$ is divisible by $p$, we can write
$r^{*}(\pi_{*}(x))=py$ with $y\in H^{k}(\widetilde{M},\Z)$.
Thus gives:
$$\widetilde{\pi}_{*}(s^{*}(x))+(p-\tors)= py.$$
Applying $\widetilde{\pi}^{*}$, we get:
$$ps^{*}(x)=p\widetilde{\pi}^{*}(y).$$
Since $H^{k}(\widetilde{X},\Z)$ is $p$-torsion-free, this is also the case for the group $s^{*}(H^{k}(X,\Z))$, hence:
$$\widetilde{\pi}^{*}(y)=s^{*}(x).$$
Hence, by applying $\widetilde{\pi}_{*}$, we get $\widetilde{\pi}_{*}(s^{*}(x))=py.$
\end{proof}

\begin{prop}\label{parrynormal}
Let $(\widetilde{X},\widetilde{G},s)$ be a pullback of $(X,G)$. Assume that $H^{*}(X,\Z)$ and $H^{*}(\widetilde{X},\Z)$ are $p$-torsion-free.
Assume that $s^{*}(H^{k}(X,\F))$ admits a complementary vector space, stable under the action of $\widetilde{G}$. Then,
if $(\widetilde{X},\widetilde{G})$ is $H^{k}$-normal then $(X,G)$ is $H^{k}$-normal.
\end{prop}
\begin{proof}
We use the notation for the Jordan decomposition as in (\ref{normality1}).
First note that since $s$ is surjective, $s^{*}: H^{k}(X,\Z)\rightarrow H^{k}(\widetilde{X},\Z)$ is injective too. Then, we can prove that $s^{*}: H^{k}(X,\F)\rightarrow H^{k}(\widetilde{X},\F)$ is injective. Since $H^{*}(X,\Z)$ and $H^{*}(\widetilde{X},\Z)$ are $p$-torsion-free, we only have to verify that $p|s^*(x)$ implies $p|x$ for all $x\in H^{k}_f(X,\Z)$. However, if $s^*(x)$ is divisible by $p$ then $s^*(x)\cdot s^*(y)=s^*(x\cdot y)$ if divisible by $p$ for all $y\in H^{n-k}_f(X,\Z)$. Since the degree of $s$ is not divisible by $p$, it follows that $p$ divides $x\cdot y$ for all $y\in H^{n-k}_f(X,\Z)$. So, by Poincaré duality, it means that $x$ is divisible by $p$.

Now, let $x\in H^{k}(X,\Z)^{G}$, we use the characterization of Proposition \ref{lalalalala}.
We assume that there is no $y\in H^{k}(X,\Z)$ such that 
$x=y+g(y)+...+g^{p-1}(y)$ and we show that $\pi_{*}(x)$ is not divisible by $p$.
By Lemma \ref{?} $\overline{x}\notin \mathcal{N}_{p}$.

Since $s^{*}: H^{k}(X,\F)\rightarrow H^{k}(\widetilde{X},\F)$ is injective and $s^{*}(H^{k}(X,\F))$ admits a complementary vector space, stable under the action of $\widetilde{G}$, $\overline{s^{*}(x)}\notin\mathcal{N}_{p}$.
Hence, by Lemma \ref{?}, there is no $z\in H^{k}(\widetilde{X},\Z)$ such that $s^{*}(x)=z+g(z)+...+g^{p-1}(z)$.
Since $(\widetilde{X},\widetilde{G})$ is $H^{k}$-normal, $\widetilde{\pi}_{*}(s^{*}(x))$ is not divisible by $p$.
Hence, by Lemma \ref{IlLeFautBien}, $r^{*}(\pi_{*}(x))$ is not divisible by $p$.
It follows that $\pi_{*}(x)$ is not divisible by $p$.
\end{proof}
We will see an example of the use of this proposition in Section \ref{SymplecticOrder3}.
\section{Application to equivariant cohomology}\label{EquivariantCohomology}
Let $X$ be a CW-complex and $G$ a group acting on $X$ permuting the cells. 
Let $EG\rightarrow BG$ be a universal $G$-bundle in the category of CW-complexes. 
Denote by $X_{G}=EG\times_{G}X$ the orbit space for the diagonal action of $G$ on the product $EG\times X$ 
and $f: X_{G}\rightarrow BG$ the map induced by the projection onto the first factor. 
The map $f$ is a locally trivial fiber bundle with typical fiber $X$ and structure group $G$.
We define the G-equivariant cohomology of $X$ by $H_{G}^{*}(X,\Z):=H^{*}(EG\times_{G}X,\Z)$.
Then, the Leray--Serre spectral sequence associated to the map $f$ gives a spectral sequence converging to the equivariant cohomology:
$$E_{2}^{p,q}:=H^{p}(G;H^{q}(X,\Z))\Rightarrow H_{G}^{p+q}(X,\Z).$$

Hence, we are interested in calculating the cohomology of our prime order group $G=\left\langle g\right\rangle$, with coefficient in a $\Z$-module.
We have the following projective resolution of $\Z$ considered as a $G$-module:
$$\xymatrix@C=10pt{...\ar[r]^{\tau}& \Z[G]\ar[r]^{\sigma} & \Z[G]\ar[r]^{\tau} & \Z[G]\ar[r]^{\epsilon} & \Z},$$
where $\tau=g-\id$, $\sigma=\id+g+...+g^{p-1}$ and $\epsilon$ is the summation map: $\epsilon(\sum_{j=0}^{p-1} \alpha_{j}g^{j})=\sum_{j=0}^{p-1}\alpha_{j}$.
If $H$ is a $\Z[G]$-module of finite rank over $\Z$, the cohomology of $G$ with coefficients in $H$ is computed as the cohomology of the complex:
$$\xymatrix@C=10pt{0\ar[r]& H\ar[r]^{\tau} & H\ar[r]^{\sigma} & H\ar[r]^{\tau} & ...}.$$
\begin{prop}\label{HEquivSpec}
Assume that $H$ is a $p$-torsion-free finitely generated $\Z$-module and a $\Z[G]$-module with $G=\left\langle g \right\rangle$ group of prime order $p\leq 19$.
Then for all $i\in\mathbb{N}^*$:
\begin{itemize}
\item[(i)]
$H^{0}(G,H)=H^{G}$,
\item[(ii)]
$H^{2i-1}(G,H)=(\Z/p\Z)^{\ell_{-}(H)}$,
\item[(iii)]
$H^{2i}(G,H)=(\Z/p\Z)^{\ell_{+}(H)}$.
\end{itemize}
\end{prop}
\begin{proof}
\begin{itemize}
\item[(i)] 
By definition, $H^{0}(G,H)=\Ker \tau$.
\item[(ii)]
First, we remark that the cohomology of the torsion part is trivial.
Since the torsion part $H_{\tors}$ is $p$-torsion-free, $p$ is always invertible in $H_{\tors}$.
Hence, all $x\in  H_{\tors}\cap \Ker \tau$ can be written $x=\frac{1}{p}(\underbrace{x+...+x}_{p\ \text{times}})\in\Image \sigma;$
and for $x\in H_{\tors}$ such that $x+g(x)+...+g^{p-1}(x)=0$, we can write $x=g(y)-y$ with $y=-\frac{1}{p}\left[(p-1)x+(p-2)g(x)+...+2g^{p-3}(x)+g^{p-2}(x)\right]$.

So, henceforth, we assume that $H$ is torsion-free. In odd degrees, $H^{2i-1}(G,H)=\Ker \sigma/ \Image \tau$.
Using the notation of Section \ref{BoissiereSartiH2}, in the proof of \cite[Theorem 74.3]{Reiner}, it is shown that:
\begin{align*}
&\Ker \sigma=\mathcal{O}_{K}b_{1}\oplus...\oplus\mathcal{O}_{K}b_{r+s-1}\oplus Ab_{r+s},\\
&\Image \tau=E_{1}b_{1}\oplus...\oplus E_{r+s-1}b_{r+s-1}\oplus E_{r+s}Ab_{r+s},
\end{align*}
with $b_{1},...,b_{n}$, $\mathcal{O}_{K}$-free elements in $\Ker \sigma$, $A$ an $\mathcal{O}_{K}$-ideal of $K$ and 
$$E_{1}=...=E_{r}=\mathcal{O}_{K},\ \ \ \ E_{r+1}=...=E_{r+s}=(\xi_{p}-1)\mathcal{O}_{K}.$$
The numbers $r$ and $s$ in the last equalities are the same as in Proposition \ref{defipro}.
Moreover, we find in the proof of \cite[Theorem 74.3]{Reiner} that $\mathcal{O}_{K}/(\xi_{p}-1)\mathcal{O}_{K}=A/(\xi_{p}-1)A=\Z/p\Z$. 
Hence, we get $\Ker \sigma/ \Image \tau=(\Z/p\Z)^{\ell_{-}^{k}(X)}$.
\item[(iii)]
For $i\geq 1$, $H^{2i}(G,H)=\Ker \tau/ \Image \sigma$.
By (iii) of Proposition \ref{defipro}:
$$H^{G}\simeq \oplus_{i=1}^{r} (\mathcal{O}_{K},a_{i})^{G}\oplus \Z^{\oplus t}.$$
By Lemma \ref{?} and Proposition \ref{defipro} (ii), all the elements in $\oplus_{i=1}^{r} (\mathcal{O}_{K},a_{i})^{G}$ can be written as $y+g(y)+...+g^{p-1}(y)$ with $y\in H$. Then, the result follows.
\end{itemize}
\end{proof}
If we apply Proposition \ref{HEquivSpec} to cohomology, we obtain:
\begin{cor}\label{equico}
Let $X$ be a topological space and $G$ be an automorphism group of prime order $p$.
Assume that $H^{*}(X,\Z)$ is $p$-torsion-free and finitely generated with $p\leq 19$. Then, for $0\leq k\leq 2\dim X$ we have:
\begin{itemize}
\item
$H^{0}(G,H^{k}(X,\Z))=H^{k}(X,\Z)^{G}$,
\item
$H^{2i-1}(G,H^{k}(X,\Z))=(\Z/p\Z)^{\ell_{-}^{k}(X)}$,
\item
$H^{2i}(G,H^{k}(X,\Z))=(\Z/p\Z)^{\ell_{+}^{k}(X)}$,
\end{itemize}
for all $i\in\mathbb{N}^{*}$.
\end{cor}
\begin{rmk}
In \cite[Section 3]{SmithTh}, a similar result is shown for cohomology with coefficients in $\F$. 
\end{rmk}

When $G$ acts freely,
$$H^{*}(X/G,\Z)\simeq H_{G}^{*}(X,\Z),$$
where the isomorphism is
induced by the natural map $f : EG\times_{G}X\rightarrow X/G$,
(see for instance \cite{Bott}). 
Hence, when the action of $G$ is free and the spectral sequence of equivariant cohomology is degenerate at the second page, Corollary \ref{equico} will provide the cohomology of $X/G$.

\section{Resolution of singularities}\label{Resolution of singularities}
Let $X$ be a compact connected orientable manifold of dimension $n$ and $G=\left\langle g\right\rangle$ be an automorphism group of prime order $p$.
With the information that $H^*_f(X/G,\Z)$ is a lattice, we have only provided an upper bound for the coefficients of surjectivity (Proposition \ref{BasicAlphaInequality2}). To obtain more information, we need to consider also the geometry of $X/G$. One way to do so is to consider a resolution of singularities.
\subsection{Exceptional lattice of a resolution}
Let $r:\widetilde{X/G}\rightarrow X/G$ be a resolution of singularities and 
consider the following diagram:
$$\xymatrix{
\ar[d]^{\pi_*}H^*(X,\Z)& \\
  H^{*}(X/G,\Z)\ar[r]^{r^*}&H^{*}(\widetilde{X/G},\Z).
 }$$
Remark first that since $r$ is surjective $r^*$ is an injective map. 
\begin{defi}\label{exceptlattice}
The sublattice 
$$N_{k,r}:=r^*\left[\pi_*\left(H^k(X,\Z)\oplus H^{n-k}(X,\Z)\right)\right]^\bot_f$$ of $H^{k}_f(\widetilde{X/G},\Z)\oplus H^{n-k}_f(\widetilde{X/G},\Z)$
will be called the $k$-th exceptional lattice of $r$.
\end{defi}

It follows that $N_{k,r}^\bot$ is a primitive over-lattice of 
$r^*\left[\pi_*\left(H^k(X,\Z)\oplus H^{n-k}(X,\Z)\right)_f\right]$.
We have the following inclusion:
$$r^*\left[\pi_*\left(H^k(X,\Z)\oplus H^{n-k}(X,\Z)\right)_f\right]\subset r^*\left[H^k_f(X/G,\Z)\oplus H^{n-k}_f(X/G,\Z)\right]\subset N_k^\bot.$$
Now, since $r^*$ is injective, it induces the following injective map of $\Z/p\Z$-vector space:
$$\frac{H^k_f(X/G,\Z)\oplus H^{n-k}_f(X/G,\Z)}{\pi_*\left(H^k(X,\Z)\oplus H^{n-k}(X,\Z)\right)_f}\hookrightarrow \frac{N_k^\bot}{r^*\left[\pi_*\left(H^k(X,\Z)\oplus H^{n-k}(X,\Z)\right)_f\right]}.$$
Then, using (\ref{BasicLatticeTheory}) and taking the logarithm in base $p$, it follows that:
\begin{align*}
&\log_p\discr \pi_*\left(H^k(X,\Z)\oplus H^{n-k}(X,\Z)\right)_f-\log_p\discr H^k_f(X/G,\Z)\oplus H^{n-k}_f(X/G,\Z)\\
&\leq \log_p \discr r^*\left[\pi_*\left(H^k(X,\Z)\oplus H^{n-k}(X,\Z)\right)_f\right]-\log_p \discr N_k^\bot.
\end{align*}
By (\ref{BasicAlphaEquality02}):
$$ \alpha_{k}(X)+\alpha_{n-k}(X)\leq\frac{\log_p \discr r^*\left[\pi_*(H^k(X,\Z)\oplus H^{n-k}(X,\Z))_f\right]-\log_p \discr N_k^\bot}{2}.$$
Moreover, since $r^*$ is injective:
\begin{equation}
\discr r^*\left[\pi_*\left(H^k(X,\Z)\oplus H^{n-k}(X,\Z)\right)_f\right]=\discr \pi_*\left(H^k(X,\Z)\oplus H^{n-k}(X,\Z)\right)_f.
\label{equationdefin2}
\end{equation}
Furthermore, $\widetilde{X/G}$ is smooth, so $H^k_f(\widetilde{X/G},\Z)\oplus H^{n-k}_f(\widetilde{X/G},\Z)$ is unimodular. Hence, by (\ref{BasicLatticeTheory2}),
$\discr N_k=\discr N_k^\bot$. It follows that
$$\alpha_{k}+\alpha_{n-k}\leq\frac{\log_p \discr \pi_*\left(H^k(X,\Z)\oplus H^{n-k}(X,\Z)\right)_f-\log_p \discr N_k}{2}.$$
Finally, by (\ref{BasicAlphaEquality}) and Proposition \ref{BasicAlphaInequality2}, we obtain:
\begin{prop}\label{BasicResolution}
Let $X$ be a compact connected orientable manifold of dimension $n$ and $G$ an automorphism group of prime order $p$. Let $r:\widetilde{X/G}\rightarrow X/G$ be a resolution of singularities. Let $N_{k,r}$
be the $k$-th exceptional lattice of $r$.
Then:
$$\alpha_{k}(X)+\alpha_{n-k}(X)\leq\rk H^k_f(X,\Z)^G-\ell_p^k(X)-\frac{1}{2}\log_p \discr N_{k,r}.$$
If moreover, $p\leq19$:
$$\alpha_{k}(X)+\alpha_{n-k}(X)\leq\ell_+^k(X)-\frac{1}{2}\log_p \discr N_{k,r}.$$
\end{prop}
\begin{rmk}\label{BasicResolutionRmk}
When $n=2m$ is even and $k=m$, the discriminant of $N_{m}$ is given by the discriminant of two copies of $N_{m}^{\frac{1}{2}}:=r^*\left[\pi_*\left(H^{m}(X,\Z)\right)\right]^\bot_f$ and the inequalities of Proposition \ref{BasicResolution} are
$$\alpha_{m}\leq\frac{\rk H^{m}_f(X,\Z)^G-\ell_p^{m}(X)-\log_p \discr N_{m}^{\frac{1}{2}}}{2};$$
and, when $p\leq 19$,
$$\alpha_{m}\leq\frac{\ell_+^{m}(X)-\log_p \discr N_{m}^{\frac{1}{2}}}{2}.$$
\end{rmk}
Now, the objective will be to express $N_k$ in term of exceptional divisors and calculate its discriminant.
\subsection{Using blow-ups}\label{blowup}
In this section $X$ is a complex manifold and $G$ an automorphism group of prime order $p$.
The simplest case that we can consider for applying Proposition \ref{BasicResolution} is the following.

Let $s:\widetilde{X}\rightarrow X$ be the blow-up of $X$ in the fixed locus $Z$ of $G$, let $\widetilde{G}$ be the automorphism group induced by $G$ on $\widetilde{X}$. When $X$ is K\"ahler, the cohomology of $\widetilde{X}$ can be calculated using \cite[Theorem 7.31]{Voisin}. 

\begin{thm}(\cite{Voisin})\label{Vois}
Let $X$ be a K\"ahler manifold and let $Z\subset X$ be a submanifold of codimension $d$. 
Let $s:\widetilde{X}\rightarrow X$ be the blow-up of $X$ in $Z$. Let $E=s^{-1}(Z)$ be the exceptional divisor and $j: E\hookrightarrow \widetilde{X}$ be the embedding. 
Let $h=c_{1}(\mathcal{O}_{E}(1))\in H^{2}(E,\Z)$. Then, we have an isomorphism of Hodge structures:
$$\xymatrix{
H^{k}(X,\Z)\oplus \left(\bigoplus_{i=0}^{d-2} H^{k-2i-2}(Z,\Z)\right)\ar[rrrr]^{\ \ \ \ \ \ \ \ \ \ \ \ \tau^{*}+\sum_{i} j_{*}\circ h^{i}\circ \tau_{|E}^{*}} & & & & H^{k}(\widetilde{X},\Z).}$$
Here, $h^{i}$ is the morphism by the cup-product with $h^{i}\in H^{2i}(E,\Z)$. 
\end{thm}
Now if $\widetilde{X}/\widetilde{G}$ is smooth, $\widetilde{X}/\widetilde{G}$ provides a resolution of singularities of $X/G$:
$$\xymatrix{
 \widetilde{X} \ar[r]^{s} \ar[d]^{\widetilde{\pi}} & X,\ar[d]^{\pi} \\
  \widetilde{X}/\widetilde{G} \ar[r]^{r}& X/G 
   }$$
 where the exceptional lattice of $r$ (introduced in Definition \ref{exceptlattice}) could be expressed in terms of the cohomology of the fixed locus $Z$ via Theorem \ref{Vois}.
So, in this section, we will study when such $\widetilde{X}/\widetilde{G}$ is smooth.
\begin{rmk}
The triple $(\widetilde{X},\widetilde{G},s)$ is a pullback of $(X,G)$ (cf. Definition \ref{pullback}).
\end{rmk}

At each fixed point of $G$, by Cartan \cite[Lemma 1]{Cartan} we can locally linearize the action of $G$. Thus, at a fixed point $x\in X$, the action of $G$ on $X$ is locally equivalent to the action of $G=\left\langle g\right\rangle$ on $\mathbb{C}^{n}$ via
$$g=\diag (\xi_{p}^{k_{1}},...,\xi_{p}^{k_{n}}),$$ where $\xi_{p}$ is a p-th root of unity.
Without loss of generality, we can assume that $k_{1}\leq...\leq k_{n}\leq p-1$. 
\begin{lemme}\label{ResolutionLemma}
Let $X$ be a complex manifold of dimension $n$ and $Z$ be a connected component of $\Fix G$, the fixed locus of an automorphism group $G$ of prime order $p$. Let $m=\dim Z$ and $E\rightarrow Z$ be the exceptional divisor of the blow-up $\widetilde{X}\rightarrow X$ of $X$ in $Z$. Let $\widetilde{G}$ be the automorphism group induced by $G$ on $\widetilde{X}$. Assume that the local action of $G$ around a point of $Z$ is given by $$g=\diag (1,...,1,\xi_{p}^{k_{m+1}},...,\xi_{p}^{k_{n}}),$$ then the different local actions of $\widetilde{G}$ around its fixed points contained in $E$ are given by the diagonal matrices:
$$\diag(1,...,1,\xi_{p}^{k_{1+m}-k_{i}},...,\xi_{p}^{k_{i-1}-k_{i}},\xi_{p}^{k_{i}},\xi_{p}^{k_{i+1}-k_{i}},...,\xi_{p}^{k_{n}-k_{i}})$$ for all $m+1\leq i\leq n$.
\end{lemme}
\begin{proof}
Let $G=\left\langle \diag (\xi_{p}^{k_{1}},...,\xi_{p}^{k_{n}})\right\rangle$ acting on $\mathbb{C}^{n}$.
Since $k_{1}=...=k_{m}=0$, we have to consider the blow-up of $\mathbb{C}^{m}\times\mathbb{C}^{n-m}$ in $\mathbb{C}^{m}\times\left\{0\right\}$,
which is given by $\mathbb{C}^{m}\times\widetilde{\mathbb{C}^{n-m}}$ where $\widetilde{\mathbb{C}^{n-m}}$ is the blow-up of $\mathbb{C}^{n-m}$ in $0$.
So, without loss of generality, we can assume that $m=0$.
 
Let $\widetilde{\mathbb{C}^{n}}$ be the blow-up of $\mathbb{C}^{n}$ in 0 and $\widetilde{G}$ the automorphism group of $\widetilde{\mathbb{C}^{n}}$ induced by $G$.
We will describe the action of $\widetilde{G}$ on 
\begin{align*}
\widetilde{\mathbb{C}^{n}}=\left\{\left((x_{1},...,x_{n}),(a_{1}:...:a_{n})\right)\in \right.&\mathbb{C}^{n}\times \mathbb{P}^{n-1}\ |\ \\ &\left.\rk\left((x_{1},...,x_{n}),(a_{1},...,a_{n})\right)=1\right\}.
\end{align*}
We denote by $$\mathscr{O}_{i}=\left\{\left((x_{1},...,x_{n}),(a_{1}:...:a_{n})\right)\in \mathbb{C}^{n}\times \mathbb{P}^{n-1}\ |\
a_{i}\neq 0\right\}$$ the chart $a_{i}\neq 0$.
We have 
\begin{align*}
\widetilde{\mathbb{C}^{n}}\cap\mathscr{O}_{i}=\left\{\left((x_{1},...,x_{n}),(a_{1},...,a_{i-1},a_{i+1},...,a_{n})\right)\in \right.& \mathbb{C}^{n}\times \mathbb{C}^{n-1}\ |\ \\
&\left.x_{j}=x_{i}a_{j},\ j\in \left\{1,...,n\right\}\right\}.
\end{align*}
Hence, we have an isomorphism:
$$\xymatrix@R0pt{
f:\widetilde{\mathbb{C}^{n}}\cap\mathscr{O}_{i}\ar[r]& \mathbb{C}^{n}\\
 \left((x_{1},...,x_{n}),(a_{1}:...:a_{n})\right)\ar[r]& (a_{1},...,a_{i-1},x_{i},a_{i+1},...,a_{n})
}$$
Thus, the action of $\widetilde{G}$ on $\widetilde{\mathbb{C}^{n}}\cap\mathscr{O}_{i}$ provides an action on $\mathbb{C}^{n}$ given by the diagonal matrix $$\diag(\xi_{p}^{k_{1}-k_{i}},...,\xi_{p}^{k_{i-1}-k_{i}},\xi_{p}^{k_{i}},\xi_{p}^{k_{i+1}-k_{i}},...,\xi_{p}^{k_{n}-k_{i}}).$$
\end{proof}
We recall the following property from the proof of \cite[Proposition 6]{Prill}:
\begin{prop}(\cite{Prill})
The quotient $\mathbb{C}^{n}/G$ is smooth if and only if $\rk g-\id=1$; that is  $k_{1}=...= k_{n-1}=0$.
\end{prop}
Hence, we deduce from Lemma \ref{ResolutionLemma} the following.
\begin{cor}\label{singu}
Let $X$ be a complex manifold of dimension $n$ and $G$ an automorphism group of prime order $p$. Let $x\in \Fix G$.
\begin{itemize}
\item[(i)]
The variety $X/G$ is smooth in $\pi(x)$ if and only if the local action around $x$ is given by $\diag (0,...,0,\xi_{p}^k)$.
\item[(ii)]
Let $\widetilde{X}$ be the blow-up of $X$ in the connected components of $\Fix G$ of codimension $\geq 2$ and $\widetilde{G}$ the automorphism group induced by $G$ on $\widetilde{X}$.

The quotient $\widetilde{X}/\widetilde{G}$ is smooth if and only if the local action of $G$ around all $x\in \Fix G$, is given by a matrix of the form:
$\diag (0,...,0,\xi_{p}^k,...,\xi_{p}^k)$.
\end{itemize}
\end{cor}
So, the local actions which are given by a diagonal matrix of the form $\diag (0,...,0,\xi_{p}^k,...,\xi_{p}^k)$ are of particular interest for our purpose of using Proposition \ref{BasicResolution}.
Keeping the same notation, we set:
\begin{defi}\label{typepoint}
A fixed point is simple if there is $i\in \left\{1,...,n\right\}$ such that $k_{1}=...= k_{i}=0$ and $k_{i+1}=...= k_{n}$.
\end{defi}
Let $X$ be a complex manifold and $G$ an automorphism group of prime order acting on $X$.
We have studied the case of a single blow-up, we now consider a sequence of blow-ups:
$$\xymatrix{
X_{k} \ar@(ul,ur)[]^{G_{k}}\ar[r]^{s_{k}} \ar[d]^{\pi_{k}}& \cdot\cdot\cdot \ar[r]^{s_{2}}& X_{1} \ar@(ul,ur)[]^{G_{1}}\ar[r]^{s_{1}}\ar[d]^{\pi_{1}}& X,\ar@(ul,ur)[]^{G}\ar[d]^{\pi}\\
 X_{k}/G_k \ar[r]^{r_{k}} &\cdot\cdot\cdot \ar[r]^{r_{2}}& X_{1}/G_1\ar[r]^{r_{1}}& X/G 
   }$$   
where each $s_{i+1}$ is the blow-up of $X_{i}$ in $\Fix G_{i}$ where $X=X_{0}$ and $G=G_{0}$.
However, in most cases, it is not possible to find $k$ such that $X_{k}/G_k $ is smooth.
\begin{prop}\label{blowup2}
There exists $k\in \mathbb{N}$ such that $X_{k}/G_k$ is smooth if and only if either all the fixed points of $G$ are simple or $p=3$. Moreover, in the case where $p=3$, $X_{2}/G_2$ is smooth.
\end{prop}
\begin{proof}
By \cite[Lemma 1]{Cartan}, we can assume that $X=\mathbb{C}^{n}$ and $$G=\left\langle \diag (\xi_{p}^{k_{1}},...,\xi_{p}^{k_{n}})\right\rangle.$$

First, if $0$ is a simple fixed point, by Corollary \ref{singu}, $X_1/G_1$ is smooth.
If $p=3$, we will show that $X_2/G_2$ is smooth.
Let $G=\left\langle \diag (\xi_{3}^{k_{1}},...,\xi_{3}^{k_{n}})\right\rangle$ acting on $\mathbb{C}^{n}$. 
Without loss of generality, we can assume that all $k_{i}$ are different from 0.
Let $\widetilde{\mathbb{C}^{n}}$ be the blow-up of $\mathbb{C}^{n}$ in 0.
By Lemma \ref{ResolutionLemma},
in the chart $a_{i}\neq 0$ the action of $G_{1}$ is given by the diagonal matrix
$$\diag(\xi_{3}^{k_{1}-k_{i}},...,\xi_{3}^{k_{i-1}-k_{i}},\xi_{3}^{k_{i}},\xi_{3}^{k_{i+1}-k_{i}},...,\xi_{3}^{k_{n}-k_{i}}).$$
As $p=3$, there is a $j$ such that $k_{1}=...=k_{j}=1$ and $k_{j+1}=...=k_{n}=2$.
So, if $i\leq j$, by permuting the $i$-th and the $j$-th coordinates of the chart $a_{i}\neq 0$, we reduce the action to the form
$$\diag(1,...,1,\xi_{3},...,\xi_{3}).$$
If $i> j$, by a permutation of coordinates we obtain
$$\diag(\xi_{3}^{2},...,\xi_{3}^{2},1,...,1).$$
In both cases, these are simple fixed points. 
Hence, all the points of $\Fix G_{1}$ are simple fixed points. Moreover, as there are no fixed points with both eigenvalues $\xi_{3}$, $\xi_{3}^{2}$ present in the diagonal matrix of the action, we can conclude that the components of $\Fix G_{1}$ with spectra $(1,...,1,\xi_{3},...,\xi_{3})$ and $(\xi_{3}^{2},...,\xi_{3}^{2},1,...,1)$ are disjoint and can be blown up independently.
Hence, by Corollary \ref{singu}, (ii), $X_2/G_2$ is smooth.

Secondly, we will show that in the other cases, $X_{k}/G_k$ will never be smooth.
We start with $\dim X=2$, so that
by \cite[Lemma 1]{Cartan}, we can assume that $X=\mathbb{C}^{2}$ and $G=\left\langle \diag(\xi_{p},\xi_{p}^{\alpha})\right\rangle$,
$p>3$ and $\alpha$ not equal to $0$ or $1$.

Let $\xi_{p}$ be a fixed non-trivial $p$-th root of the unity, for $x\in \Fix G_{i}$, we can write:
$$(X_{i},G_{i},x)\sim (\mathbb{C}^{2},\left\langle \diag(\xi_{p},\xi_{p}^{\beta})\right\rangle,0),$$ 
where $\sim$ means that the local actions are the same.
Hence, we can define a sequence as follows:
$$u_{i}=\left\{\beta\in\Z/p\Z |\ \exists x\in X_{i}:\ (X_{i},G_{i},x)\sim (\mathbb{C}^{2},\left\langle \diag(\xi_{p},\xi_{p}^{\beta})\right\rangle,0)\right\}.$$

For instance, $u_{0}=\left\{\alpha,\frac{1}{\alpha}\right\}$, $u_{1}=\left\{\alpha-1,\frac{1}{\alpha-1},\frac{\alpha}{1-\alpha},\frac{1-\alpha}{\alpha}\right\}$,...
Now assume that there is $i\in\mathbb{N}$ such that $X_{i}/G_i$ is smooth.
Let $i$ be the smallest integer such that $X_{i}/G_i$ is smooth.
By Corollary \ref{singu}, (i), we have $u_{i}=\left\{0\right\}$ and we can write $u_{i-1}=\left\{\alpha_{1},...,\alpha_{k}\right\}$.
Let $x\in \Fix G_{i-1}$ such that $(X_{i-1},G_{i-1},x)\sim (\mathbb{C}^{2},\left\langle \diag(\xi_{p},\xi_{p}^{\alpha_{j}})\right\rangle,0)$, with $\alpha_{j}\in u_{i-1}\setminus \left\{0\right\}$.
Let $\widetilde{\mathbb{C}^{2}}$ be the blow-up of $\mathbb{C}^{2}$ in 0.
The action of $\left\langle \diag(\xi_{p},\xi_{p}^{\alpha_{j}})\right\rangle$ on $\widetilde{\mathbb{C}^{2}}$ has two fixed points $a_{1}$ and $a_{2}$ with (see Lemma \ref{ResolutionLemma}):
$$(\widetilde{\mathbb{C}^{2}},\left\langle \diag(\xi_{p},\xi_{p}^{\alpha_{j}})\right\rangle,a_{1})\sim (\mathbb{C}^{2},\diag(\xi_{p},\xi_{p}^{\alpha_{j}-1}),0)$$
and
$$(\widetilde{\mathbb{C}^{2}},\left\langle \diag(\xi_{p},\xi_{p}^{\alpha_{j}})\right\rangle,a_{2})\sim (\mathbb{C}^{2},\diag(\xi_{p}^{\alpha_{j}},\xi_{p}^{1-\alpha_{j}}),0).$$
Hence, $\alpha_{j}-1\in u_{i}$,
but $u_{i}=\left\{0\right\}$,
so, necessarily, $\alpha_{j}=1$ and so $u_{i-1}=\left\{1\right\}$.

We do the same calculation with $u_{i-2}=\left\{\alpha_{1}',...,\alpha_{k}'\right\}$.
Let $x\in \Fix G_{i-2}$ be such that:
$$(X_{i-2},G_{i-2},x)\sim \left(\mathbb{C}^{2},\left\langle \diag(\xi_{p},\xi_{p}^{\alpha_{j}'})\right\rangle,0\right).$$
with $\alpha_{j}'\in u_{i-2}\setminus \left\{0,1\right\}$. We remark that $u_{i-2}\setminus \left\{0,1\right\}$ is not empty because $X_{i-1}/G_{i-1}$ is not smooth by definition of $i$.
Let $\widetilde{\mathbb{C}^{2}}$ be the blow-up of $\mathbb{C}^{2}$ in 0.
The action of $\left\langle \diag(\xi_{p},\xi_{p}^{\alpha_{j'}})\right\rangle$ on $\widetilde{\mathbb{C}^{2}}$ has 2 fixed points $a_{1}'$ and $a_{2}'$ with (see Lemma \ref{ResolutionLemma}):
$$\left(\widetilde{\mathbb{C}^{2}},\left\langle \diag(\xi_{p},\xi_{p}^{\alpha_{j}'})\right\rangle,a_{1}\right)\sim \left(\mathbb{C}^{2},\diag(\xi_{p},\xi_{p}^{\alpha_{j}'-1}),0\right)$$
and
$$\left(\widetilde{\mathbb{C}^{2}},\left\langle \diag(\xi_{p},\xi_{p}^{\alpha_{j}'})\right\rangle,a_{1}\right)\sim \left(\mathbb{C}^{2},\diag(\xi_{p}^{\alpha_{j}'},\xi_{p}^{1-\alpha_{j}'}),0\right).$$
But $$\left(\mathbb{C}^{2},\diag(\xi_{p}^{\alpha_{j}'},\xi_{p}^{1-\alpha_{j}'}),0\right)\sim \left(\mathbb{C}^{2},\diag(\xi_{p},\xi_{p}^{\frac{1-\alpha_{j}'}{\alpha_{j}'}}),0\right).$$
Hence, necessarily, $\alpha_{j}'=2$ and $\alpha_{j}'=\frac{1}{2}$. In $\Z/p\Z$, $2=\frac{1}{2}$ if, and only if, $p=3$.
Hence, $p=3$ and we are done.
~\\

Now, we assume $n>2$.
By \cite[Lemma 1]{Cartan}, we can assume that $X=\mathbb{C}^{n}$ and $G=\left\langle \diag(\xi_{p}^{k_{1}},\xi_{p}^{k_{2}},...,\xi_{p}^{k_{n}})\right\rangle$. Without loss of generality, we can assume that all the $k_{i}$ are different from $0$.
Moreover, by assumption, $p>3$ and the $k_{i}$ are not all equal. Without loss of generality, we can assume that $k_{1}=1$. Since not all the $k_{i}$ are equal, there is $j\in \left\{1,...,n\right\}$ such that $k_{j}\neq 1$. We also denote $\alpha= k_{j}$. 
We denote $X'=\mathbb{C}^{2}$ and $G'=\left\langle \diag(\xi_{p},\xi_{p}^{\alpha})\right\rangle$. 
And we define as before the sequence:
$$u'_{i}=\left\{\beta\in\Z/p\Z |\ \exists x\in X'_{i}:\ (X'_{i},G'_{i},x)\sim (\mathbb{C}^{2},\left\langle \diag(\xi_{p},\xi_{p}^{\beta})\right\rangle,0)\right\}.$$
We also define the following sequence:
\begin{align*}
U_{i}=&\left\{\beta\in\Z/p\Z |\ \exists x\in X_{i}:\ \right.\\
&\left. (X_{i},G_{i},x)\sim (\mathbb{C}^{n},\left\langle \diag(\xi_{p},\xi_{p}^{\beta},\xi_{p}^{t_{3}},...,\xi_{p}^{t_{n}})\right\rangle,0)\right\}.
\end{align*}
We have to show that $U_{i}\neq\left\{0\right\}$ for all $i\in\mathbb{N}$.
Indeed, $U_{i}\supset u'_{i}$, and we have seen that $u'_{i}\neq\left\{0\right\}$ for all $i\in\mathbb{N}$. 
\end{proof}
\section{Proof of Theorem \ref{Maincor} and outcomes}\label{The main theorem}
Theorem \ref{Maincor} is a direct consequence of the following result which will be proven in this section.
\begin{thm}\label{main}
Let $X$ be a compact K\"ahler manifold of complex dimension $n$ and $G=\left\langle g\right\rangle$ an automorphism group of prime order $p\leq 19$. Let $c:=\codim \Fix G$.
Assume that
\begin{itemize}
\item[(i)]
$H^*(X,\Z)$ and  $H^*(\Fix G,\Z)$ are $p$-torsion-free,
\item[(ii)]
$h^*(\Fix G,\Q)\geq\sum_{1\leq q < p} \ell_{q}^*(X),$
\item[(iii)]
all fixed points of $G$ are simple,
\item[(iv)]
$c\geq\frac{n}{2}+1$.
\end{itemize}
Then, for all $2n-2c+1\leq k \leq2c-1$,
$(X,G)$ is $H^k$-normal.
\end{thm}
We first explain why Theorem \ref{Maincor} follows from Theorem \ref{main}.
Boissi\`ere, Nieper-Wisskirchen and Sarti in \cite[Corollary 4.3]{SmithTh} have proven the following formula:
\begin{prop}\label{BNSFormula}
Let $X$ be a compact connected orientable manifold of dimension $n$ and $G$ an automorphism group of prime order $p$.
If the spectral sequence of equivariant cohomology with coefficients in $\F$ degenerates at the second page, then:
$$h^*(\Fix G,\F)=\sum_{1\leq q < p} \ell_{q}^*(X).$$
\end{prop}
Hence, since we assume that $H^*(\Fix G,\Z)$ is $p$-torsion-free, we obtain Theorem \ref{Maincor} from Theorem \ref{main} and Proposition \ref{BNSFormula}.

Before proving Theorem \ref{main} in Section \ref{ProofMain}, we discuss its hypotheses in few words. 
\begin{rmk}
The condition $p\leq 19$ is necessary to work with the spectral sequence of equivariant cohomology with coefficients in $\Z$ and use Corollary \ref{equico} (see Lemma \ref{MainProofLemma3}). 
However, it is possible to avoid the condition $p\leq 19$ working, under some additional assumptions, with the spectral sequence of equivariant cohomology with coefficients in $\F$. 
I made the choice not to provide such a theorem, since in practice it is quite rare to have to work with $p>19$; 
for instance, until now, none symplectic automorphism of such order is known on irreducible symplectic manifolds.
\end{rmk}
\begin{rmk}\label{RemarkCodim}
The condition on the codimension of $\Fix G$ in Theorem \ref{main} could be improved. However, hypothesis (ii) would have to be replaced by a more 
elaborate condition. 
\end{rmk}
\begin{rmk}
As explained in Section \ref{blowup}, to avoid the condition that the fixed points of $G$ are simple, we need to find another practical way of resolving the singularities of $X/G$. We could also use two blow-ups in the case $p=3$. We provide such an example in Section \ref{SymplecticOrder3}. 
\end{rmk}
\subsection{Proof of Theorem \ref{main}}\label{ProofMain}
The idea of the proof is to calculate the discriminant of the exceptional lattice $N_{k,r}$, where $r$ is obtained from a blow-up, and apply Proposition  \ref{BasicResolution}.

All the fixed points of $G$ are simple, so we recall from Section \ref{blowup} that we are in the configuration of the following commutative diagram:
\begin{equation}
\xymatrix{
\widetilde{X} \ar[r]^{s} \ar[d]^{\widetilde{\pi}} & X\ar[d]^{\pi}\\
 \widetilde{X}/\widetilde{G} \ar[r]^{r}& X/G,    
   }
\label{cacacommutant0}  
\end{equation}
$\widetilde{G}$ be the automorphism group induced by $G$ and $\widetilde{X}/\widetilde{G}$ is smooth.
We also denote $M=X/G$, $\widetilde{M}=\widetilde{X}/\widetilde{G}$, $V=X\smallsetminus \Fix G$, $U=\pi(V)$, $F=s^{-1}(\Fix G)$, we use the same symbol $F$ for its image by $\widetilde{\pi}$ and we adopt the notation introduced in Section \ref{BoissiereSartiH2} and Section \ref{DefinitionIntro}.
Diagram (\ref{cacacommutant0}) induces a commutative diagram on cohomology:
$$\xymatrix{
H^{k}(M,\Z)\ar@/_/[r]_{\pi^{*}} \ar[d]_{r^{*}}& H^{k}(X,\Z)\ar@/_/[l]^{\pi_{*}}\ar[d]^{s^{*}}\\
H^{k}(\widetilde{M},\Z)\ar@/^/[r]^{\widetilde{\pi}^{*}}& \ar@/^/[l]_{\widetilde{\pi}_{*}} H^{k}(\widetilde{X},\Z).
}$$
\begin{lemme}\label{commut2}
Let $0\leq k \leq 2n$ and $x\in H^{k}(X,\Z)$. 
Then:
\begin{itemize}
\item[(i)]
$\widetilde{\pi}_{*}(s^{*}(x))=r^{*}(\pi_{*}(x))+(p-\tors),$
\item[(ii)]
$N_{k,r}=\widetilde{\pi}_{*}(s^{*}(H^{k}(X,\Z)\oplus H^{2n-k}(X,\Z)))^\bot_f$.
\end{itemize}
\end{lemme}
\begin{proof}
We remark that $(\widetilde{X},\widetilde{G},s)$ is a pull-back of $(X,G)$, so (i) follows from Lemma \ref{IlLeFautBien}.
Applying (i) to $H^{k}_f(\widetilde{M},\Z)$, we obtain:
\begin{equation}
\widetilde{\pi}_{*}(s^{*}(H^{k}(X,\Z))_f= r^{*}(\pi_*(H^{k}(X,\Z))_f).
\label{ExceptionalLattice1}
\end{equation}
Then, it follows (ii) by definition of $N_{k,r}$.
\end{proof}
\begin{lemme}\label{LemmeMainProof5}
\begin{itemize}
\item[(i)]
For all $i\leq 2c-2$,
$H^i(X,\Z)=H^i(V,\Z)$,
\item[(ii)]
$H^{2c-1}(V,\Z)$ is $p$-torsion-free.
\end{itemize}
\end{lemme}
\begin{proof}
We have the following exact sequence:
$$\xymatrix{ H^i(X,V,\Z)\ar[r] & H^{i}(X,\Z)\ar[r] & H^{i}(V,\Z)\ar[r]&H^{i+1}(X,V,\Z).}$$
Moreover,
by Thom's isomorphism for each connected component $S_m$ of $\Fix G$:
$H^i(X,X\smallsetminus S_m,\Z)=H^{i-2\codim S_m}(S_m,\Z).$
Hence:
$$H^i(X,V,\Z)=\sum_{S_m\subset \Fix G} H^i(X,X\smallsetminus S_m,\Z)=0,$$
for all $i\leq 2c-1$.
So, we get (i).
Moreover, we obtain the following exact sequence:
\begin{equation}
\xymatrix{ 0\ar[r] & H^{2c-1}(X,\Z)\ar[r] & H^{2c-1}(V,\Z)\ar[r]&H^{2c}(X,V,\Z).}
\label{EquationMainProof4}
\end{equation}
From the hypotheses, we know that $H^{2c-1}(X,\Z)$ and $H^*(\Fix G,\Z)$ are $p$-torsion-free;
so, by Thom's isomorphism, $H^{2c}(X,V,\Z)$ is $p$-torsion-free. 
Then, (ii) follows from (\ref{EquationMainProof4}).
\end{proof}
We consider the following commutative diagram:
\begin{equation}
\xymatrix@C=10pt{ \ar[d]^{d\widetilde{\pi}^{*}}H^{k}(\mathscr{N}_{\widetilde{M}/F},\mathscr{N}_{\widetilde{M}/F}-0,\Z)=H^{k}(\widetilde{M},U,\Z)\ar[r]^{\ \ \ \ \ \ \ \ \ \ \ \ \ \ \ \ \ \ \ \ g}&\ar[d]_{\widetilde{\pi}^{*}}H^{k}(\widetilde{M},\Z)\\
H^{k}(\mathscr{N}_{\widetilde{X}/F},\mathscr{N}_{\widetilde{X}/F}-0,\Z)=H^{k}(\widetilde{X},V,\Z)\ar[r]^{\ \ \ \ \ \ \ \ \ \ \ \ \ \ \ \ \ \ \ \ h}&H^{k}(\widetilde{X},\Z),
}
\label{Thom}
\end{equation}
where $\mathscr{N}_{\widetilde{X}/F}-0$ and $\mathscr{N}_{\widetilde{M}/F}-0$ are vector bundles minus the zero section.
We denote $T_k:=h(H^{k}(\widetilde{X},V,\Z))$.
\begin{lemme}\label{*}
When $2n-2c+1\leq k \leq2c-1$,
we have: 
$$H^k(\widetilde{X},\Z)=s^*(H^k(X,\Z))\oplus T_k.$$
Moreover:
$$T_k\bot\ s^*(H^{2n-k}(X,\Z)).$$
\end{lemme}
\begin{proof}
By Thom's isomorphism, $H^{k}(\widetilde{X},V,\Z)= H^{k-2}(F,\Z)$, and the map $h$ can be identified with the morphism 
$j_{*}:H^{k-2}(F,\Z)\rightarrow H^{k}(\widetilde{X},\Z)$, where $j$ is the inclusion in $\widetilde{X}$.
As in the proof of \cite[Theorem 7.31]{Voisin}, the map
$$(s^{*},j_{*}): H^{k}(X,\Z)\oplus H^{k-2}(F,\Z)\rightarrow H^{k}(\widetilde{X},\Z)$$ 
is surjective and its kernel is the image of the map
$$\bigoplus_{S_{d}\subset \Fix G} H^{k-2r_{d}}(S_{d},\Z)\rightarrow H^{k}(X,\Z)\oplus H^{k-2}(F,\Z),$$
where $r_{k}$ is the codimension of the component $S_{d}$ of $\Fix G$.
But in our case $\bigoplus_{S_{d}\subset \Fix G} H^{k-2r_{d}}(S_{d},\Z)=0$.

The orthogonality of this sum follows from the projection formula.
\end{proof}
Moreover, Theorem \ref{Vois} also shows that:
$$\rk T_k= \rk \bigoplus_{S_{d}\subset \Fix G}\bigoplus_{i=0}^{r_{d}-2} H^{k-2i-2}(S_{d},\Z),$$
where $r_{d}$ is the codimension of the component $S_{d}$ of $\Fix G$.
Hence, when $2n-2c+1\leq k \leq2c-1$, 
we obtain:
\begin{equation}
\rk T_k=h^{2*+\epsilon}(\Fix G,\Z),
\label{EquationMainProof5}
\end{equation}
where $\epsilon=1$ if $k$ is odd and $\epsilon=0$ if $k$ is even.
\begin{lemme}\label{MainProofLemma}
We have:
\begin{itemize}
\item[(i)]
When $2n-2c+1\leq k \leq2c-1$:
$$\discr \widetilde{\pi}_*(T_k\oplus T_{2n-k})=p^{2h^{2*+\epsilon}(\Fix G,\Z)},$$
where $\epsilon=1$ if $k$ is odd and $\epsilon=0$ if $k$ is even.
\item[(ii)]
$N_{k,r}$ is the primitive over-lattice of $\widetilde{\pi}_*(T_k\oplus T_{2n-k})$.
\end{itemize}
\end{lemme}
\begin{proof}
By Lemma \ref{*}, we have:
\begin{equation}
H^{k}(\widetilde{X},\Z)\oplus H^{2n-k}(\widetilde{X},\Z)=s^{*}\left(H^{k}(X,\Z)\oplus H^{2n-k}(X,\Z)\right)\oplus^{\bot} \left(T_k\oplus T_{2n-k}\right).
\label{EquationMainProof8}
\end{equation}
Since $\widetilde{X}$ and $X$ are smooth, $H^{k}(\widetilde{X},\Z)\oplus H^{2n-k}(\widetilde{X},\Z)$ and $H^{k}(X,\Z)\oplus H^{2n-k}(X,\Z)$ are unimodular. 
It follows that $T_k\oplus T_{2n-k}$ is unimodular.
Moreover, by Thom's isomorphism $H^{k}(\widetilde{X},V,\Z)= H^{k-2}(F,\Z)$, and the map $h$ can be identified with the morphism 
$j_{*}:H^{k-2}(F,\Z)\rightarrow H^{k}(\widetilde{X},\Z)$, where $j$ is the inclusion in $\widetilde{X}$.
It follows that $T_k\subset H^k(\widetilde{X},\Z)^{\widetilde{G}}$.
Hence, by Lemma \ref{commut} (ii), we obtain:
$$\discr \widetilde{\pi}_*(T_k\oplus T_{2n-k})=p^{\rk T_k\oplus T_{2n-k}}.$$
From (\ref{EquationMainProof5}), when $2n-2c+1\leq k \leq2c-1$,  
$$\rk T_k=h^{2*+\epsilon}(\Fix G,\Z),$$
where $\epsilon=1$ if $k$ is odd and $\epsilon=0$ if $k$ is even.
Remark that if $2n-2c+1\leq k \leq2c-1$ then $2n-2c+1\leq 2n-k \leq2c-1$.
Hence, $\rk T_k=\rk T_{2n-k}=h^{2*+\epsilon}(\Fix G,\Z)$.
So, it follows (i).

Now taking the image of (\ref{EquationMainProof8}) by $\widetilde{\pi}_*$, we obtain
$\widetilde{\pi}_*(s^{*}\left(H^{k}(X,\Z)\oplus H^{2n-k}(X,\Z)\right))\oplus^\bot \widetilde{\pi}_*(\left(T_k\oplus T_{2n-k}\right))$ as a sublattice of full rank of $H^{k}(\widetilde{M},\Z)\oplus H^{2n-k}(\widetilde{M},\Z)$.
Moreover, applying Proposition \ref{SmithProp}, the direct sum remains orthogonal. 
So, (ii) follows from Lemma \ref{commut2} (ii). 
\end{proof}

The Thom class of $\mathscr{N}_{\widetilde{M}/F}$ is sent by $d\widetilde{\pi}^{*}$ to $p$ times the Thom class of $\mathscr{N}_{\widetilde{X}/F}$. This provides:
: $$d\widetilde{\pi}^{*}(H^{k}(\mathscr{N}_{\widetilde{M}/F},\mathscr{N}_{\widetilde{M}/F}-0,\Z))=pH^{k}(\mathscr{N}_{\widetilde{X}/F},\mathscr{N}_{\widetilde{X}/F}-0,\Z).$$
Then, by commutativity of diagram (\ref{Thom}) and Proposition \ref{SmithProp}, we have
$g(H^{k}(\widetilde{M},U,\Z))=\widetilde{\pi}_{*}(T_k)$.
It follows the exact sequence:
\begin{equation}
\xymatrix{ 0\ar[r] &\widetilde{\pi}_{*}(T_k) \ar[r] & H^{k}(\widetilde{M},\Z)\ar[r] & H^{k}(U,\Z).}
\label{ExactSequence1}
\end{equation}
In order to calculate the discriminant of $N_{k,r}$ which is by Lemma \ref{MainProofLemma} (ii), the primitive over lattice of $\widetilde{\pi}_{*}(T_k)\oplus \widetilde{\pi}_{*}(T_{2n-k})$,
we need to know the $p$-torsion of $H^{k}(U,\Z)\oplus H^{2n-k}(U,\Z)$.
Let $H^k_p(U,\Z)$ be the $\F$-vector spaces given by the $p$-torsion part of $H^k(U,\Z)$.

To lighten a bit the notation, in the following calculations, we denote $\ell_*^{k}$ for $\ell_*^{k}(X)$ and $N_k$ for $N_{k,r}$.
\begin{lemme}\label{MainProofLemma3}
Let $k \leq2c-1$,
we have
if $k=2a$:
$$\dim_{\F} H^k_p(U,\Z)\leq \sum_{j=0}^{a-1} \ell_+^{2j}+\sum_{j=0}^{a-1} \ell_-^{2j+1}.$$
If $k=2a+1$:
$$\dim_{\F} H^k_p(U,\Z)\leq \sum_{j=0}^{a-1} \ell_+^{2j+1}+\sum_{j=0}^{a} \ell_-^{2j}.$$
\end{lemme}
\begin{proof}
Since $G$ acts freely on $V$, we recall from Section \ref{EquivariantCohomology} that
$$H^{*}(U,\Z)\simeq H_{G}^{*}(V,\Z).$$

Using the second page of the spectral sequence of equivariant cohomology, it follows that
\begin{equation}
\dim_{\F} H^k_p(U,\Z)\leq \sum_{i=1}^k \dim_{\F}H^i(G,H^{k-i}(V,\Z)).
\label{EquationMainProof9}
\end{equation}
The group $H^0(G,H^{k}(V,\Z)$ does not provide any contribution to $H^k_p(U,\Z)$ because 
by Corollary \ref{equico}, $H^0(G,H^{k}(V,\Z)=H^{k}(V,\Z)^G$ and 
by Lemma \ref{LemmeMainProof5}, $H^{k}(V,\Z)$ is $p$-torsion-free.

Moreover, by Lemma \ref{LemmeMainProof5} (i), (\ref{EquationMainProof9}) becomes:
$$\dim_{\F} H^k_p(U,\Z)\leq \sum_{i=1}^k \dim_{\F}H^i(G,H^{k-i}(X,\Z)).$$
Hence, by Corollary \ref{equico},
if $k=2a$:
$$\dim_{\F} H^k_p(U,\Z)\leq \sum_{j=0}^{a-1} \ell_+^{2j}+\sum_{j=0}^{a-1} \ell_-^{2j+1}.$$
If $k=2a+1$:
$$\dim_{\F} H^k_p(U,\Z)\leq \sum_{j=0}^{a-1} \ell_+^{2j+1}+\sum_{j=0}^{a} \ell_-^{2j}.$$
\end{proof}
From (\ref{ExactSequence1}) and Lemma \ref{MainProofLemma} (ii):
$$\dim_{\F}\left(\frac{N_k}{\widetilde{\pi}_*(T_k\oplus T_{2n-k})}\right)\leq \dim_{\F} H^k_p(U,\Z) \oplus H^{2n-k}_p(U,\Z).$$
When $2n-2c+1\leq k \leq2c-1$, we can apply Lemma \ref{MainProofLemma3} to $H^k_p(U,\Z)$ and to $H^{2n-k}_p(U,\Z)$. When $k=2a$,
$$\dim_{\F}\left(\frac{N_k}{\widetilde{\pi}_*(T_k\oplus T_{2n-k})}\right)\leq \sum_{j=0}^{a-1} \ell_+^{2j}+\sum_{j=0}^{a-1} \ell_-^{2j+1}+\sum_{j=0}^{n-a-1} \ell_+^{2j}+\sum_{j=0}^{n-a-1} \ell_-^{2j+1}.$$
Hence, by Lemma \ref{MainProofLemma} (i) and (\ref{BasicLatticeTheory}):
\begin{equation}
\log_p\discr N_k\geq 2h^{2*}(\Fix G,\Z)-2\left[\sum_{j=0}^{a-1} \ell_+^{2j}+\sum_{j=0}^{a-1} \ell_-^{2j+1}+\sum_{j=0}^{n-a-1} \ell_+^{2j}+\sum_{j=0}^{n-a-1} \ell_-^{2j+1}\right].
\label{EquationMainProof10}
\end{equation}
Similarly, when $k=2a+1$, we obtain:
\begin{equation}
\log_p\discr N_k\geq 2h^{2*+1}(\Fix G,\Z)-2\left[\sum_{j=0}^{a-1} \ell_+^{2j+1}+\sum_{j=0}^{a} \ell_-^{2j}+\sum_{j=0}^{n-a-2} \ell_+^{2j+1}+\sum_{j=0}^{n-a-1} \ell_-^{2j}\right].
\label{EquationMainProof11}
\end{equation}
Now the strategy to finish the proof is the following. Equations (\ref{EquationMainProof10}) and (\ref{EquationMainProof11}) alone are not practical but, as we will see, their sum can be simplified. Hence, the objective is to use (\ref{EquationMainProof10})+(\ref{EquationMainProof11}) to prove our theorem. Notice first that since $c\geq\frac{n}{2}+1$, the interval $\left\{2n-2c+1,...,2c-1\right\}$ always has  cardinal at least 3. Hence, we can always find $k$ such that $k, k+1\in\left\{2n-2c+1,...,2c-1\right\}$. 
Using (\ref{EquationMainProof10})+(\ref{EquationMainProof11}) we will provide an upper bound for $\log_p\discr N_k+\log_p\discr N_{k+1}$ which will provide the $H^k\oplus H^{k+1}$-normality. Then, the claim follows covering $\left\{2n-2c+1,...,2c-1\right\}$ by pairs of consecutive numbers $\left\{k,k+1\right\}$.

So, let $k,k+1\in \left\{2n-2c+1,...,2c-1\right\}$, we can assume that $k=2a$, if $k$ is odd the proof is identical.
Hence, summing (\ref{EquationMainProof10})+(\ref{EquationMainProof11}) and applying Proposition \ref{BasicAlphaInequality2} (ii), we have:
\begin{align*}
\log_p\discr N_{k}+\log_p\discr N_{k+1}-2h^{*}(\Fix G,\Z)&\geq-2\left[\sum_{j=0}^{k-1} \ell_+^{j}+\sum_{j=0}^{k} \ell_-^{j}+\sum_{j=0}^{2n-k-2} \ell_+^{j}+\sum_{j=0}^{2n-k-1} \ell_-^{j}\right]\\
&\geq-2\left[\sum_{j=0}^{k-1} \ell_+^{j}+\sum_{j=0}^{k} \ell_-^{j}+\sum_{j=0}^{2n-k-2} \ell_+^{2n-j}+\sum_{j=0}^{2n-k-1} \ell_-^{2n-j}\right]\\
&\geq-2\left[\sum_{j=0}^{k-1} \ell_+^{j}+\sum_{j=0}^{k} \ell_-^{j}+\sum_{j=k+2}^{2n} \ell_+^{j}+\sum_{j=k+1}^{2n} \ell_-^{2n-j}\right]\\
&\geq-2\left[\sum_{1\leq q < p} \ell_{q}^*-\ell_+^k-\ell_+^{k+1}\right].
\end{align*}
Hence, by hypothesis,
\begin{equation}
\log_p\discr N_{k}+\log_p\discr N_{k+1}\geq2\left[\ell_+^k+\ell_+^{k+1}\right],
\label{numerodefin}
\end{equation}
and, by Proposition \ref{BasicResolution},
$$\alpha_k=\alpha_{2n-k}=\alpha_{k+1}=\alpha_{2n-k-1}=0.$$
\begin{rmk}\label{remarkdefin}
We have proved a bit more than what is claimed in Theorem \ref{main}. Under its hypotheses, we have seen that $\widetilde{\pi}_{*}(s^{*}(H^{k}(X,\Z)\oplus H^{2n-k}(X,\Z)))$ is primitive, in $H^*(\widetilde{M},\Z)$, for all $2n-2c+1\leq k \leq2c-1$.

Indeed, from (\ref{numerodefin}) and Proposition \ref{BasicResolution}, $ \discr N_{k}=p^{\ell_+^k}$. 
Moreover, from (\ref{equationdefin2}) and Corollary \ref{MainBasicCor}, $\discr r^{*}(\pi_{*}(H^{k}_f(X,\Z)\oplus H^{2n-k}_f(X,\Z)))=p^{\ell_+^k}$.
So, by (\ref{ExceptionalLattice1}), we have:
\begin{equation} 
\log_p\discr N_{k}=\discr\widetilde{\pi}_{*}(s^{*}(H^{k}_f(X,\Z)\oplus H^{2n-k}_f(X,\Z))). 
\label{RemarkPrimitivity}
\end{equation}
Since $\widetilde{M}$ is smooth, by (\ref{BasicLatticeTheory2}), we have $\discr N_{k}=\discr N_{k}^\bot$.
However, by Lemma \ref{commut2} (ii), $\widetilde{\pi}_{*}(s^{*}(H^{k}(X,\Z)\oplus H^{2n-k}(X,\Z)))\subset N_{k}^\bot$. So by (\ref{RemarkPrimitivity}), 
$\widetilde{\pi}_{*}(s^{*}(H^{k}(X,\Z)\oplus H^{2n-k}(X,\Z)))=N_{k}^\bot$, which is primitive in $H^*(\widetilde{M},\Z)$.
\end{rmk}
\subsection{Refinement on the codimension of the fixed locus}
We provide an extension of Theorem \ref{Maincor} when $\Fix G$ is a bit larger. As mentioned in Remark \ref{RemarkCodim}, we will need more technical conditions. However, when $\codim \Fix G=\left\lceil \frac{n}{2}\right\rceil$, it is still possible to provide a result which is interesting in practice (see for instance the application in \cite[Section 15]{Kapfer}). We will see another example in Section \ref{IntegralBasisHilbertScheme}.
\begin{thm}\label{Main2}
Let $X$ be a compact K\"ahler manifold of complex dimension $n$ and $G=\left\langle g\right\rangle$ be an automorphism group of prime order $p\leq 19$. 
Assume that:
\begin{itemize}
\item[(i)]
$H^*(X,\Z)$ and  $H^*(\Fix G,\Z)$ are $p$-torsion-free.
\item[(ii)]
$h^{2*}(\Fix G,\Q)\geq \ell_{+}^{2*}(X)+\ell_-^{2*+1}(X)$ if $n$ is even and
$h^{2*+1}(\Fix G,\Q)\geq \ell_{+}^{2*+1}(X)+\ell_-^{2*}(X)$ if $n$ is odd.
\item[(iii)]
All fixed points of $G$ are simple.
\item[(iv)]
$\codim \Fix G=\left\lceil \frac{n}{2}\right\rceil$.
\item[(v)]
When $n$ is even, let $\Sigma$ be the component of $\Fix G$ of dimension $\frac{n}{2}$. We assume that $\Sigma$ is connected.
Let $\left[\Sigma\right]$ be the cohomology class associated to $\Sigma$.
We assume that $\left[\Sigma\right]$ is not of torsion (in particular non 0) and is primitive in $H^{n}(X,\Z)$.
\end{itemize}
Then,
$(X,G)$ is $H^n$-normal.
\end{thm}
This section is dedicated to the proof of this theorem.
The proof is different when $n$ is even or odd. When $n$ is odd, it is a straightforward consequence of the results in the previous section. When $n$ is even we have to do more work.
Let us first prove the theorem when $n$ is odd.
\subsubsection*{Case: $n$ is odd}
We have $n=2m+1$ and $\codim \Fix G=m+1$, so $n\leq 2m+1$.
It follows that Lemmas \ref{commut2}, \ref{LemmeMainProof5}, \ref{*}, \ref{MainProofLemma} remain true. The problem is that the interval $\left\{2n-2c+1,...,2c-1\right\}$ has only cardinality one; hence it is no longer possible to sum (\ref{EquationMainProof10}) and (\ref{EquationMainProof11}).
However, (\ref{EquationMainProof11}) remains true and becomes
$$\log_p\discr N_n\geq 2h^{2*+1}(\Fix G,\Z)-2\left[\sum_{j=0}^{m-1} \ell_+^{2j+1}+\sum_{j=0}^{m} \ell_-^{2j}+\sum_{j=0}^{m-1} \ell_+^{2j+1}+\sum_{j=0}^{m} \ell_-^{2j}\right].$$
Applying Proposition \ref{BasicAlphaInequality2}, it follows:
$$\log_p\discr N_n\geq 2h^{2*+1}(\Fix G,\Z)-2\left[\sum_{j=0}^{m-1} \ell_+^{2j+1}+\sum_{j=0}^{m} \ell_-^{2j}+\sum_{j=0}^{m-1} \ell_+^{2n-2j-1}+\sum_{j=0}^{m} \ell_-^{2n-2j}\right].$$
So:
$$\log_p\discr N_n\geq 2h^{2*+1}(\Fix G,\Z)-2\left[\ell_{+}^{2*+1}+\ell_-^{2*}-\ell_{+}^{2m+1}\right].$$
Then, we get the result using hypothesis (ii) and Proposition \ref{BasicResolution}.
\subsubsection*{Case: $n$ is even}
In this case $n=2m$ and $2m>2c-1$.
We have to provide a variation of each lemma of Section \ref{ProofMain} which are not true anymore. 
Let us start with Lemma \ref{LemmeMainProof5}.
\begin{lemme}\label{LemmaOtherProof1}
\begin{itemize}
\item[(i)]
$H^{2m-1}(X,\Z)=H^{2m-1}(V,\Z)$,
\item[(ii)]
$H^{2m}(V,\Z)$ is $p$-torsion-free.
\end{itemize}
\end{lemme}
\begin{proof}
We have the following exact sequence:
$$\xymatrix{ H^{2m-1}(X,V,\Z)\ar[r] & H^{2m-1}(X,\Z)\ar[r] & H^{2m-1}(V,\Z)\ar[r]&H^{2m}(X,V,\Z)\ar[r]^{\rho}&H^{2m}(X,\Z).}$$
By Thom's isomorphism 
$$H^{2m-1}(X,V,\Z)=0\ \text{and}\ T: H^{2m}(X,V,\Z)\simeq H^0(\Sigma,\Z).$$
In \cite[Section 11.1.2]{Voisin}, it is explained that the cohomology class associated to $\Sigma$ 
can be expressed as $\left[\Sigma\right]=\rho(T^{-1}(1))$.
By assumption $\left[\Sigma\right]$ is not of torsion,
it follows that 
$\rho$ is injective.
So, $$H^{2m-1}(X,\Z)=H^{2m-1}(V,\Z).$$
Moreover, we obtain the following exact sequence:
$$
\xymatrix{ 0\ar[r]&H^{0}(\Sigma,\Z)\ar[r]^{\rho\circ T^{-1}} & H^{2m}(X,\Z)\ar[r] & H^{2m}(V,\Z)\ar[r]&H^{2m+1}(X,V,\Z).}
$$
From the hypotheses, we know that $H^{2m}(X,\Z)$ and $H^*(\Fix G,\Z)$ are $p$-torsion-free;
so by Thom's isomorphism $H^{2m+1}(X,V,\Z)$ is $p$-torsion-free. 
Moreover, we assumed that $\left[\Sigma\right]=\rho(T^{-1}(1))$ is primitive.
Hence, it follows (ii).
\end{proof}
We consider the following commutative diagram:
\begin{equation}
\xymatrix@C=10pt{ \ar[d]^{d\widetilde{\pi}^{*}}H^{2m}(\mathscr{N}_{\widetilde{M}/F},\mathscr{N}_{\widetilde{M}/F}-0,\Z)=H^{2m}(\widetilde{M},U,\Z)\ar[r]^{\ \ \ \ \ \ \ \ \ \ \ \ \ \ \ \ \ \ \ \ g}&\ar[d]_{\widetilde{\pi}^{*}}H^{2m}(\widetilde{M},\Z)\\
H^{2m}(\mathscr{N}_{\widetilde{X}/F},\mathscr{N}_{\widetilde{X}/F}-0,\Z)=H^{2m}(\widetilde{X},V,\Z)\ar[r]^{\ \ \ \ \ \ \ \ \ \ \ \ \ \ \ \ \ \ \ \ h}&H^{2m}(\widetilde{X},\Z),
}
\label{ThomII}
\end{equation}
where $\mathscr{N}_{\widetilde{X}/F}-0$ and $\mathscr{N}_{\widetilde{M}/F}-0$ are vector bundles minus the zero section.
We denote $R:=h(H^{2m}(\widetilde{X},V,\Z))$.
We prove the following variation of Lemma \ref{*}:
\begin{lemme}\label{LemmaOtherProof4}
We can write:
$$H^{2m}(\widetilde{X},\Z)=s^{*}(H^{2m}(X,\Z))\oplus^{\bot} T,$$
with $R=T\oplus \Z[\Sigma]$. 
\end{lemme}
\begin{proof}
As in proof of Lemma \ref{*}, the map $h$ can be identified with the morphism $j_{*}:H^{2m-2}(F,\Z)\rightarrow H^{2m}(\widetilde{X},\Z)$; moreover, the map
$$(s^{*},j_{*}): H^{2m}(X,\Z)\oplus H^{2m-2}(F,\Z)\rightarrow H^{2m}(\widetilde{X},\Z)$$ 
is surjective and its kernel is the image of the map
$$\bigoplus_{S_{d}\subset \Fix G} H^{2m-2r_{d}}(S_{d},\Z)\rightarrow H^{2m}(X,\Z)\oplus H^{2m-2}(F,\Z),$$
where $r_{d}$ is the codimension of the component $S_{d}$ of $\Fix G$.
However, in our case $$\bigoplus_{S_{d}\subset \Fix G} H^{2m-2r_{d}}(S_{d},\Z)=H^{0}(\Sigma,\Z).$$
So, we obtain
$H^{2m}(\widetilde{X},\Z)=s^{*}(H^{2m}(X,\Z))\oplus T,$ and the 
orthogonality of the sum can also be deduced using the projection formula.
\end{proof}
Let $\widetilde{R}$ be the minimal primitive over-lattice of $\widetilde{\pi}_{*}(R)$ in $H^{2m}(\widetilde{M},\Z)$ and $\widetilde{T}$ the minimal primitive over-lattice of $\widetilde{\pi}_{*}(T)$ in $H^{2m}(\widetilde{M},\Z)$.
From Lemmas \ref{LemmaOtherProof4} and \ref{commut2} (ii), by the same proof as Lemma \ref{MainProofLemma} (ii):
\begin{equation}
N_{2m}=\widetilde{T}^{2\oplus}.
\label{NTMainProof2}
\end{equation}
As in the previous proof, we need to calculate the discriminant of $\widetilde{T}$. 
By the property of the Thom isomorphism, $$d\widetilde{\pi}^{*}(H^{2m}(\mathscr{N}_{\widetilde{M}/F},\mathscr{N}_{\widetilde{M}/F}-0,\Z))=pH^{2m}(\mathscr{N}_{\widetilde{X}/F},\mathscr{N}_{\widetilde{X}/F}-0,\Z).$$
Then, by commutativity of diagram (\ref{Thom}) and Proposition \ref{SmithProp}, we have
$g(H^{2m}(\widetilde{M},U,\Z))=\widetilde{\pi}_{*}(R)$.
It follows the exact sequence:
\begin{equation}
\xymatrix{ 0\ar[r] &\widetilde{\pi}_{*}(R) \ar[r] & H^{2m}(\widetilde{M},\Z)\ar[r] & H^{2m}(U,\Z).}
\label{ExactSequence2}
\end{equation}
With the same proof as Lemma \ref{MainProofLemma3}, we obtain from Lemma \ref{LemmaOtherProof1}:
\begin{equation}
\dim_{\F} H^{2m}_p(U,\Z)\leq \sum_{j=0}^{m-1} \ell_+^{2j}+\sum_{j=0}^{m-1} \ell_-^{2j+1}.
\label{EquationOtherProof3}
\end{equation}
However, (\ref{ExactSequence2}) will provide only an upper bound for $\dim_{\F}\frac{\widetilde{R}}{\widetilde{\pi}_{*}(R)}$. We need some lemmas to compare $\frac{\widetilde{R}}{\widetilde{\pi}_{*}(R)}$ with $\frac{\widetilde{T}}{\widetilde{\pi}_{*}(T)}$.
\begin{lemme}\label{R}
There exists $x\in \widetilde{\pi}_{*}(T)$ such that $\frac{x+(-1)^{n-1}\widetilde{\pi}_{*}(s^{*}(\Sigma))}{p}\in H^{2m}(\widetilde{M},\Z)$.
\end{lemme}
\begin{proof} 
Let $s_{1}:Y\rightarrow X$ be the blow-up of $X$ in $\Sigma$ and $\Sigma_{1}$ the exceptional divisor, and $s_{2}:\widetilde{X}\rightarrow Y$ the blow-up in the other components of $\Fix G$ such that $s=s_{2}\circ s_{1}$. We denote $\Sigma_{2}=s_{2}^{*}(\Sigma_{1})$ and $\widetilde{\Sigma}=\widetilde{\pi}_*(\Sigma_{2})$.
Consider the following diagram:
$$\xymatrix{
\Sigma_{2}\ar[d]^{g_{2}}\ar@{^{(}->}[r]^{l_{2}}& \widetilde{X} \ar[d]^{s_{2}}\\
 \Sigma_{1}\ar[d]^{g_{1}}\ar@{^{(}->}[r]^{l_{1}}& Y \ar[d]^{s_{1}}\\
   \Sigma \ar@{^{(}->}[r]^{l_{0}}  & X,
   }$$
where $l_{0}$, $l_{1}$ and $l_{2}$ are the inclusions and $g_{i}:=s_{i|\Sigma_{i}}$, $i\in \left\{1,2\right\}$.

We have $\widetilde{\pi}_{*}(\mathcal{O}_{\widetilde{X}})=\mathcal{O}_{\widetilde{M}}\oplus\mathcal{L}$, with $\mathcal{L}^{p}=\mathcal{O}_{\widetilde{M}}\left(-\left(\sum_{S_{k}\subset \Fix G\setminus \Sigma} \widetilde{S_{k}}\right)-\widetilde{\Sigma}\right)$,
where each $\widetilde{S_{k}}$ is the exceptional divisor associated to the connected component $S_{k}\neq \Sigma$ of $\Fix G$.
Thus, $$\frac{\sum_{S_{k}\subset \Fix G\setminus \Sigma} \widetilde{S_{k}}+\widetilde{\Sigma}}{p}\in H^{2}(\widetilde{M},\Z).$$
It follows that $$\left(\frac{\sum_{S_{k}\subset \Fix G\setminus \Sigma} \widetilde{S_{k}}+\widetilde{\Sigma}}{p}\right)^{m}\in H^{2m}(\widetilde{M},\Z).$$
By Lemma \ref{commut} (i), we get:
\begin{equation}
\frac{x+\widetilde{\pi}_{*}(\Sigma_{2}^{m})}{
p}\in H^{2m}(\widetilde{M},\Z),
\label{l'eq}
\end{equation}
with $x\in \widetilde{\pi}_{*}(T)$.

Now, it remains to calculate $\Sigma_{2}^{m}$.
By \cite[Proposition 6.7]{Fulton}, we have
$$s_{1}^{*}l_{0*}(\Sigma)=l_{1*}(c_{m-1}(E)),$$
where $E:=g_{1}^{*}(\mathscr{N}_{\Sigma/X})/\mathscr{N}_{\Sigma_{1}/Y}$.
So
\begin{align*}
s_{1}^{*}l_{0*}(\Sigma)=&l_{1*}\left(\sum_{i=0}^{m-1} (-1)^{m-1-i}c_{i}(g_{1}^{*}(\mathscr{N}_{\Sigma/X}))\cdot c_{1}(\mathscr{N}_{\Sigma_{1}/Y})^{m-1-i}\right)\\
=&l_{1*}\left(\sum_{i=1}^{m-1} (-1)^{m-1-i}c_{i}(g_{1}^{*}(\mathscr{N}_{\Sigma/X}))\cdot c_{1}(\mathscr{N}_{\Sigma_{1}/Y})^{m-1-i}\right)\\
&+(-1)^{m-1}l_{1*}\left(c_{1}(\mathscr{N}_{\Sigma_{1}/Y})^{m-1}\right)\\
=&l_{1*}\left(\sum_{i=1}^{m-1} (-1)^{m-1-i}c_{i}(g_{1}^{*}(\mathscr{N}_{\Sigma/X}))\cdot c_{1}(\mathscr{N}_{\Sigma_{1}/Y})^{m-1-i}\right)\\
&+(-1)^{m-1}\Sigma_{1}^{m}.
\end{align*}
By applying $s_{2}^{*}$, we get:
$$\Sigma_{2}^{m}=(-1)^{m-1}\left(s^{*}(\Sigma)-s_{2}^{*}l_{1*}(a)\right),$$
where $a= \sum_{i=1}^{m-1} (-1)^{m-1-i}c_{i}(g_{1}^{*}(\mathscr{N}_{\Sigma/X}))\cdot c_{1}(\mathscr{N}_{\Sigma_{1}/Y})^{m-1-i}\in T$.
And pushing forward via $\widetilde{\pi}_{*}$, we get:
$$\widetilde{\pi}_{*}(\Sigma_{2}^{m})=(-1)^{m-1}\left(\widetilde{\pi}_{*}(s^{*}(\Sigma))-\widetilde{\pi}_{*}\left(s_{2}^{*}l_{1*}(a)\right)\right).$$
The result follows from (\ref{l'eq}). 
\end{proof}
\begin{lemme}\label{LemmaOtherProof6}
We have:
$$\dim_{\F} \frac{\widetilde{T}}{\widetilde{\pi}_{*}(T)}\leq \dim_{\F}H_p^{2m}(U,\Z)-1,$$
where $H_p^{2m}(U,\Z)$ is the $p$-torsion part of $H^{2m}(U,\Z)$.
\end{lemme}
\begin{proof}
From (\ref{ExactSequence2}):
$$\dim_{\F} \frac{\widetilde{R}}{\widetilde{\pi}_{*}(R)}\leq \dim_{\F}H_p^{2m}(U,\Z).$$
But by Lemma \ref{R}, we already know that there exists $x\in \widetilde{\pi}_{*}(T)$ such that $\frac{x+(-1)^{m-1}\widetilde{\pi}_{*}(s^{*}(\Sigma))}{p}\in H^{2m}(\widetilde{M},\Z)$.

We are going to deduce $\dim_{\F}\widetilde{T}/\widetilde{\pi}_{*}(T)$ from $\dim_{\F}\widetilde{R}/\widetilde{\pi}_{*}(R)$.
If $\widetilde{\pi}_{*}(s^{*}(\Sigma))$ is divisible by $p$ in $H^{2m}(\widetilde{M},\Z)$, then  $\overline{\frac{\widetilde{\pi}_{*}(s^{*}(\Sigma))}{p}}\in (\widetilde{R}/\widetilde{\pi}_{*}(R))\setminus(\widetilde{T}/\widetilde{\pi}_{*}(T))$, if not $\overline{\frac{x+(-1)^{m-1}\widetilde{\pi}_{*}(s^{*}(\Sigma))}{
p}}\in (\widetilde{R}/\widetilde{\pi}_{*}(R))\setminus(\widetilde{T}/\widetilde{\pi}_{*}(T))$. 
In both cases: $$\dim_{\F}\frac{\widetilde{T}}{\widetilde{\pi}_{*}(T)}=\dim_{\F}\frac{\widetilde{R}}{\widetilde{\pi}_{*}(R)}-1.$$
\end{proof}
So, from Lemma \ref{LemmaOtherProof6} and (\ref{EquationOtherProof3}):
\begin{equation}
\dim_{\F}\frac{\widetilde{T}}{\widetilde{\pi}_{*}(T)}\leq\sum_{j=0}^{m-1} \ell_+^{2j}+\sum_{j=0}^{m-1} \ell_-^{2j+1}-1.
\label{EquationOtherProofTruc}
\end{equation}
From Theorem \ref{Vois} and Lemma \ref{LemmaOtherProof4}:
$$\rk T= \rk \bigoplus_{S_{d}\subset \Fix G}\bigoplus_{i=0}^{r_{d}-2} H^{2m-2i-2}(S_{d},\Z),$$
where $r_{d}$ is the codimension of the component $S_{d}$ of $\Fix G$.
Hence, 
we obtain:
\begin{equation}
\rk T=h^{2*}(\Fix G,\Z)-2,
\label{EquationOtherProof1}
\end{equation}
With the same proof as Lemma \ref{MainProofLemma}, it follows:
\begin{equation}
\discr \widetilde{\pi}_*(T)=p^{h^{2*}(\Fix G,\Z)-2}.
\label{EquationOtherProof2}
\end{equation}
Combining (\ref{EquationOtherProof2}), (\ref{EquationOtherProofTruc}) and (\ref{BasicLatticeTheory}), then applying Proposition \ref{BasicAlphaInequality2}, we have
\begin{align*}
\log_p\discr \widetilde{T}-h^{2*}(\Fix G,\Z)&\geq -2\left[\sum_{j=0}^{m-1} \ell_+^{2j}+\sum_{j=0}^{m-1} \ell_-^{2j+1}\right]\\
&\geq-\left[\sum_{j=0}^{m-1} \ell_+^{2j}+\sum_{j=0}^{m-1} \ell_-^{2j+1}+\sum_{j=0}^{m-1} \ell_+^{2n-2j}+\sum_{j=0}^{m-1} \ell_-^{2n-2j-1}\right]\\
&\geq-\ell_{+}^{2*}-\ell_-^{2*+1}+\ell_{+}^{2m}.
\end{align*}
Hence, using hypothesis (ii):
$$\log_p\discr \widetilde{T}\geq \ell_{+}^{2m}.$$

Finally, we conclude with (\ref{NTMainProof2}) and Remark \ref{BasicResolutionRmk}.
\begin{rmk}\label{remarkdefin2}
Same remark as Remark \ref{remarkdefin}, we have also proved that 
$\widetilde{\pi}_{*}(s^{*}(H^{n}(X,\Z)))$ is primitive in $H^*(\widetilde{M},\Z)$.
\end{rmk}
\section{Examples of applications}\label{application}
\subsection{Simply connected surfaces}
Boissi\`ere, Sarti and Nieper-Wisskirchen have proved the following (\cite[Proposition 4.5]{SmithTh}):
\begin{prop}
Let $X$ be a simply connected compact surface and $G$ be an automorphism group with at least a fixed point. Then, the spectral sequence of equivariant cohomology with coefficients in $\F$ is degenerate a the second page.
\end{prop} 
Then, from this proposition and Theorem \ref{Maincor}, we obtain the following corollary.
\begin{cor}
Let $X$ be a simply connected compact K\"ahler surface and $G$ be an automorphism group of prime order $p\leq 19$. 
Assume that $\Fix G$ is finite but not empty and contains simple points.
Then,
$(X,G)$ is $H^2$-normal.
\end{cor}
This result can be applied for example to a K3 surface endowed with a symplectic involution.

\subsection{The case of the $K3^{[2]}$-type manifolds}\label{K3typeApplication}
Boissi\`ere, Sarti and Nieper-Wisskirchen also proved in \cite[Theorem 1.1]{SmithTh} that for $X$ a 
$K3^{[2]}$-type manifold and $G$ an automorphism group of prime order $p\notin\left\{2,5,23\right\}$,
the spectral sequence of equivariant cohomology with coefficients in $\F$ is degenerate a the second page. 

Together with Theorem \ref{Maincor}, this implies:
\begin{cor}\label{CorHilb2}
Let  $X$ be a $K3^{[2]}$-type manifold and $G$ be an automorphism group of prime order $p\notin\left\{2,5,23\right\}$.
Let $c:=\codim \Fix G$.
Assume that:
\begin{itemize}
\item[(i)]
all fixed points of $G$ are simple;
\item[(ii)]
$c\geq3$.
\end{itemize}
Then,
$(X,G)$ is $H^2$ and $H^4$-normal.
\end{cor}
We remark that $H^*(X,\Z)$ is torsion-free by \cite{Markman} and $H^*(\Fix G,\Z)$ is torsion-free because $\Fix G$ contains only isolated points or smooth complex curves.
So, $H^4$-normality follows from Theorem \ref{Maincor}. It remains to show the $H^2$-normality. To do so, we will use Proposition \ref{Hkt}.

\begin{lemme}\label{Hktprop}
Let $X$ be a topological space. Let $t$ and $k$ be integers and $p$ a prime number. Assume that $H^{*}(X,\Z)$ is $p$-torsion-free. We denote by $\overline{x}\in H^{*}(X,\F)$ the reduction modulo $p$ of an element $x\in H^{*}(X,\Z)$.

If the cup product map $\Sym^t H^{k}(X,\mathbb{Q})\rightarrow H^{kt}(X,\mathbb{Q})$ is an isomorphism and $\frac{H^{kt}(X,\Z)}{\Sym^t H^{k}(X,\Z)}$ is $p$-torsion-free, then:
\begin{align*}
\mathscr{S}:\ &\Sym^t H^{k}(X,\F)\rightarrow H^{kt}(X,\F)\\
&\overline{x_{1}}\otimes...\otimes \overline{x_{t}}\mapsto \overline{x_{1}\cdot...\cdot x_{t}}
\end{align*}
is an isomorphism.
\end{lemme}\label{tkt}
\begin{proof}
\begin{itemize}
\item[]
\hspace{-0.90cm}We prove the injectivity.

Let $\overline{x_{1}}\otimes...\otimes \overline{x_{t}}\in \Sym^t H^{k}(X,\F)$ such that $\overline{x_{1}\cdot...\cdot x_{t}}=0$.
Then, there exists $y\in H^{kt}(X,\Z)$ such that $x_{1}\cdot...\cdot x_{t}=py$.
Hence, $\dot{y}\in H^{kt}(X,\Z)/\Sym^t H^{k}(X,\Z)$ is a $p$-torsion element (here $\dot{y}$ is the class of $y$ modulo $\Sym^t H^{k}(X,\Z)$). Hence, by the hypothesis $\dot{y}=0$.
It follows that $y\in \Sym^t H^{k}(X,\Z)$, so $y=y_{1}\cdot...\cdot y_{t}$ with $y_{i}\in H^{k}(X,\Z)$.
Since $\Sym^t H^{k}(X,\mathbb{Q})$ $\rightarrow H^{kt}(X,\mathbb{Q})$ is injective,
$x_{1}\otimes... \otimes x_{t}=py_{1}\otimes...\otimes y_{t}$.
So, $\overline{x_{1}}\otimes...\otimes \overline{x_{t}}=0$.
\item[]
We prove the surjectivity.

Let $\overline{y}\in H^{kt}(X,\F)$, with $y\in H^{kt}(X,\Z)$.
Since $\Sym^t H^{k}(X,\mathbb{Q})\rightarrow H^{kt}(X,\mathbb{Q})$ is an isomorphism,
there is $q\in\mathbb{N}$ and $x_{1}\otimes... \otimes x_{t}\in \Sym^t H^{k}(X,$ $\mathbb{Z})$ such that $\frac{1}{q}x_{1}\cdot... \cdot x_{t}=y$.
Hence, $\dot{y}\in H^{kt}(X,\Z)/\Sym^t H^{k}(X,\Z)$ is a $q$-torsion element. 
But since $H^{kt}(X,\Z)/\Sym^t H^{k}(X,\Z)$ is $p$-torsion-free, $p$ does not divide $q$.
And $\mathscr{S}(\frac{1}{\overline{q}}\overline{x_{1}}\otimes...\otimes \overline{x_{t}})=\overline{y}$.
\end{itemize}
\end{proof}
When $X$ is a $K3^{[2]}$-type manifold,
we know from \cite{Verbitsky} that $\Sym^2 H^{2}(X,\mathbb{Q})\rightarrow H^{4}(X,\mathbb{Q})$ is an isomorphism. Moreover, from \cite[Proposition 6.6]{SmithTh}, we know that:
$$H^{4}(X,\Z)/\Sym^2 H^{2}(X,\Z)=(\Z/2\Z)^{23}\oplus(\Z/5\Z).$$
It follows from the previous lemma that for $p\notin\left\{2,5\right\}$, 
\begin{equation}
\Sym^2 H^{2}(X,\F)\simeq H^{4}(X,\F).
\label{Sym2}
\end{equation}
So, by Proposition \ref{Hkt}, we obtain the following corollary which finishes the proof of Corollary \ref{CorHilb2}.
\begin{cor}\label{H2H4}
Let $X$ be a  $K3^{[2]}$-type manifold and $G$ be an automorphism group of prime order $p\notin\left\{2,5\right\}$.
If $(X,G)$ is $H^4$-normal, $(X,G)$ is $H^2$-normal.
\end{cor}

We describe the following example as an application of Corollary \ref{CorHilb2}. First, we recall that an automorphism on a Hilbert scheme of points $S^{[n]}$ on a K3 surface $S$ is called \emph{natural} if it is the automorphism induced by an automorphism on the K3 surface $S$.
\begin{ex}
Let $X$ be a Hilbert scheme of 2 points on a K3 surface and $G$ be a natural non-symplectic automorphism group of order 3
on $X$. Then, $(X,G)$ is $H^4$ and $H^2$-normal.
\end{ex} 
\begin{proof}
Let $G$ be such an isomorphism group on $S^{[2]}$. Since the action on the K3 surface is non-symplectic, the fixed points will be simple. Moreover, from \cite[Section 6.1]{ClassiSarti}, 
$\Fix G$ consists in 3 isolated points and 3 rational curves. 
\end{proof}

\begin{rmk}
An analogous of Corollary \ref{CorHilb2} could also be stated when $p=2$ or $5$ and
$X$ is a Hilbert scheme of 2 points on a K3 surface and $G$ a natural 
automorphism group using statement (2) of \cite[Theorem 1.1]{SmithTh}.
\end{rmk}
\subsection{Integral basis of the Hilbert scheme of two points on a surface}\label{IntegralBasisHilbertScheme}
In this section, we complete the work started in \cite[Section 7]{Kapfer}. 
Let $A$ be a smooth compact K\"ahler surface with $H^*(A,\Z)$ 2-torsion-free and $A^{[2]}$ the Hilbert scheme of two points. It can be constructed as follows:
Consider the direct product $A\times A$. Denote 
$$b:\Bl_{\Delta}(A\times A)\rightarrow A\times A$$
the blow-up along the diagonal $h:\Delta\simeq A$ with exceptional divisor $E$ (for simplicity we denote by $x$ the pull-back $h^*(x)$ for all $x\in H^*(A,\Z)$). Let $f$ be a fiber of the map $E\rightarrow \Delta$. 
Let $j:E\hookrightarrow\Bl_{\Delta}(A\times A)$, $g:\Delta\hookrightarrow A\times A$ be the embeddings 
and $p_1:A\times A\rightarrow A$, $p_2:A\times A\rightarrow A$ the projections. We also denote by $A$ the class which generates $H^{0}(A,\Z)$ and by $pt$ the class which generates $H^{4}(A,\Z)$. We have the following commutative diagram:
\begin{equation}
\xymatrix{
 E\ar[d]^{b_{|E}}\ar@{^{(}->}[rr]^{j \ \ \ \ \ }& & \Bl_{\Delta}(A\times A)\ar[d]^{b}\\
   \Delta\ar[dr]^{h} \ar[ddr]^{h}\ar@{^{(}->}[rr]^{g}  & &\ar[dl]^{p_1}\ar[ddl]^{p_2}A\times A\\
  & A & \\
  & A. &
   }
   \label{UnDiagramTechnique}
\end{equation}
The action of $\mathfrak{S}_2$ extends to an action on $\Bl_{\Delta}(A\times A)$.
Then, $A^{[2]}=\Bl_{\Delta}(A\times A)/\mathfrak{S}_2$ and we denote by $\pi: \Bl_{\Delta}(A\times A)\rightarrow A^{[2]}$ the quotient map. Moreover, from \cite[Theorem 2.2]{Totaro}, we know that $H^{*}(A^{[2]},\Z)$ is 2-torsion-free.
We want to prove the following proposition:
\begin{prop}\label{BasisHilbertH2Main}
Let $A$ be a smooth compact K\"ahler surface with $H^*(A,\Z)$ 2-torsion-free. 
The integral cohomology of $A^{[2]}$ can be expressed as follows:
\begin{itemize}
\item[(i)]
$H^1(A^{[2]},Z) = \pi_*(b^*(H^1(A\times A,\Z)))$,
\item[(ii)]
$H^2(A^{[2]},Z) = \pi_*(b^*(H^2(A\times A,\Z)))\oplus\Z\frac{\pi_*(E)}{2}$,
\item[(iii)]
$H^3(A^{[2]},Z) = \pi_*(b^*(H^3(A\times A,\Z)))\oplus\frac{1}{2}\pi_*j_*b_{|E}^*(H^1(\Delta,\Z))$,
\item[(iv)]
The lattices $\pi_*(b^*(H^4(A\times A,\Z)))$ and  $\pi_*j_*b_{|E}^*(H^2(\Delta,\Z))$ are primitive in $H^4(A^{[2]},\Z)$.
Moreover:
$$Q:=\frac{H^4(A^{[2]},\Z)}{\pi_*(b^*(H^4(A\times A,\Z)))\oplus \pi_*j_*b_{|E}^*(H^2(\Delta,\Z))}=(\Z/2\Z)^{b_2(A)}$$
and there exists an isometry $\theta$ of $H^{2}(A,\Z)$ such that the classes $\frac{\pi_*j_*b_{|E}^*(\theta(x))+\pi_*b^{*}(x\otimes x)}{2}$ generate $Q$.
\item[(v)]
$H^5(A^{[2]},Z) = \pi_*(b^*(H^5(A\times A,\Z)))\oplus\pi_*j_*b_{|E}^*(H^3(\Delta,\Z))$,
\item[(vi)]
$H^6(A^{[2]},Z) = \pi_*(b^*(H^2(A\times A,\Z)))\oplus\pi_*(f)$,
\item[(vii)]
$H^7(A^{[2]},Z) = \pi_*(b^*(H^7(A\times A,\Z)))$.
\end{itemize}
\end{prop}
This section is dedicated to the proof of this proposition.
Statements (iii) and (v) were proven in \cite[Proposition 7.1]{Kapfer}. Before proving the remaining statements, we will see that they are also interesting because they provide the ring structure of $H^*(A^{[2]},\Z)$. 
\begin{lemme}\label{IntersectionHilb2}
\begin{itemize}
\item[(i)]
For all $x_1,x_2,y_1,y_2$ in $H^{*}(A,\Z)$:
$$\left(\pi_*b^{*}(x_1\otimes y_1)\right)\cdot \left(\pi_*b^{*}(x_2\otimes y_2)\right)=\pi_*b^*\left((x_1\cdot x_2)\otimes(y_1\cdot y_2)+(x_1\cdot y_2)\otimes(y_1\cdot x_2)\right).$$
\item[(ii)]
For all $x,y,z$ in $H^{*}(A,\Z)$:
$$\left(\pi_*b^{*}(x\otimes y)\right)\cdot \left(\pi_*j_*b_{|E}^*(z)\right)=2\pi_*j_*b_{|E}^*(x\cdot y \cdot z).$$
\item[(iii)]
$\left(\frac{1}{2}\pi_*j_*(E)\right)^2=\frac{1}{2}\left(\pi_*j_*b_{|E}^*(c_1(A))-\pi_*b^*(\Delta)\right)$.
\item[(iv)]
For all $x,y$ in $H^{*}(A,\Z)$:
$$\left(\pi_*j_*b_{|E}^*(x)\right)\cdot\left(\pi_*j_*b_{|E}^*(y)\right)=2\left(\pi_*j_*b_{|E}^*(x\cdot y \cdot c_1(A))-\pi_*b^*g_*(x\cdot y)\right).$$
\end{itemize}
\end{lemme}
\begin{proof}
\begin{itemize}
\item[(i)]
Using lemma \ref{commut} (ii), we have:
\begin{align*}
\left(\pi_*b^{*}(x_1\otimes y_1)\right)\cdot \left(\pi_*b^{*}(x_2\otimes y_2)\right)&=\frac{1}{4}\left(\pi_*b^{*}(x_1\otimes y_1+y_1\otimes x_1)\right)\cdot \left(\pi_*b^{*}(x_2\otimes y_2+y_2\otimes x_2)\right)\\
&=\frac{1}{2}\pi_*\left(\left(b^{*}(x_1\otimes y_1+y_1\otimes x_1)\right)\cdot \left(b^{*}(x_2\otimes y_2+y_2\otimes x_2)\right)\right)\\
&=\frac{1}{2}\pi_*b^{*}\left((x_1\otimes y_1+y_1\otimes x_1)\cdot (x_2\otimes y_2+y_2\otimes x_2)\right)\\
&=\pi_*b^*\left((x_1\cdot x_2)\otimes(y_1\cdot y_2)+(x_1\cdot y_2)\otimes(y_1\cdot x_2)\right).
\end{align*}
\item[(ii)]
Using Lemma \ref{commut} (ii), then the projection formula and finally the commutativity of diagram (\ref{UnDiagramTechnique}), we have:
\begin{align*}
\left(\pi_*b^{*}(x\otimes y)\right)\cdot \left(\pi_*j_*b_{|E}^*(z)\right)&=\frac{1}{2}\left(\pi_*b^{*}(x\otimes y+y\otimes x)\right)\cdot \left(\pi_*j_*b_{|E}^*(z)\right)\\
&=\pi_*\left(\left(b^{*}(x\otimes y+y\otimes x)\right)\cdot \left(j_*b_{|E}^*(z)\right)\right)\\
&=\pi_*j_*\left(\left(j^*b^{*}(x\otimes y+y\otimes x)\right)\cdot \left(b_{|E}^*(z)\right)\right)\\
&=\pi_*j_*b_{|E}^*\left(g^*\left(x\otimes y+y\otimes x\right)\cdot z\right).
\end{align*}
By definition: $x\otimes y=p_1^*(x)\cdot p_2^*(y).$
It follows:
$g^*(x\otimes y)=h^*(x\cdot y)$
and by convention, we do not write the $h^*$.
So:
$$\left(\pi_*b^{*}(x\otimes y)\right)\cdot \left(\pi_*j_*b_{|E}^*(z)\right)=2\pi_*j_*b_{|E}^*\left(x\cdot y \cdot z\right).$$
\item[(iii)]
By \cite[Proposition 6.7]{Fulton}, we have:
$$b^{*}(\Delta)=j_*(c_{1}(\mathscr{F})),$$
where $\mathscr{F}:=b_{|E}^{*}(\mathscr{N}_{\Delta/A\times A})/\mathscr{N}_{E/\Bl_{\Delta}(A\times A)}$.
Moreover, $j_*(\mathscr{N}_{E/\Bl_{\Delta}(A\times A)})=j_*(E)^2$; it follows:
$$b^*(\Delta)=j_*(b_{|E}^{*}(c_1(\mathscr{N}_{\Delta/A\times A})))-j_*(E)^2.$$
Furthermore, we remark that $$c_1(\mathscr{N}_{\Delta/A\times A})=c_1\left(\frac{T_{A\times A|\Delta}}{T_{\Delta}}\right)=c_1(T_{\Delta})=h^*c_1(T_{A}),$$
so:
$$j_*(E)^2=j_*(b_{|E}^*(c_1(A)))-b^*(\Delta).$$
Then, the result follows from Lemma \ref{commut} (ii). 
\item[(iv)]
By (ii) and then by (iii), (i):
\begin{align*}
\left(\pi_*j_*b_{|E}^*(x)\right)\cdot\left(\pi_*j_*b_{|E}^*(y)\right)&=\frac{1}{4}\left(\pi_*b^{*}(x\otimes1)\right)\cdot (\pi_*j_*(E))^2\cdot\left(\pi_*b^{*}(y\otimes1)\right)\\
&=\frac{1}{2}\left(\pi_*b^{*}((x\cdot y)\otimes 1+x\otimes y)\right)\cdot \left(\pi_*j_*(b_{|E}^*(c_1(A)))-\pi_*b^*(\Delta)\right).
\end{align*}
Then, again by (ii):
\begin{equation}
\left(\pi_*j_*b_{|E}^*(x)\right)\cdot\left(\pi_*j_*b_{|E}^*(y)\right)=\frac{1}{2}\left(4\pi_*j_*b_{|E}^{*}(x\cdot y\cdot c_1(A))- \pi_*b^{*}((x\cdot y)\otimes 1+x\otimes y)\cdot\pi_*b^*(\Delta)\right).
\label{GrosseEquationIntersection}
\end{equation}
Now, we calculate $\pi_*b^{*}((x\cdot y)\otimes 1+x\otimes y)\cdot\pi_*b^*(\Delta)$ separately:
$$\pi_*b^{*}((x\cdot y)\otimes 1+x\otimes y)\cdot\pi_*b^*(\Delta)=\frac{1}{2}\pi_*b^{*}((x\cdot y)\otimes 1+1\otimes(x\cdot y) +x\otimes y+y\otimes x)\cdot\pi_*b^*(\Delta).$$
By Lemma \ref{commut} (ii), then by projection formula and finally by the definition $x\otimes y=p_1^*(x)\cdot p_2^*(y)$, we have:
\begin{align*}
\pi_*b^{*}((x\cdot y)\otimes 1+x\otimes y)\cdot\pi_*b^*(\Delta)&=\pi_*b^{*}\left(\left[(x\cdot y)\otimes 1+1\otimes(x\cdot y) +x\otimes y+y\otimes x\right]\cdot\Delta\right)\\
&=\pi_*b^{*}g_*g^*\left((x\cdot y)\otimes 1+1\otimes(x\cdot y) +x\otimes y+y\otimes x\right)\\
&=4\pi_*b^{*}g_*(x\cdot y).
\end{align*}
So, by (\ref{GrosseEquationIntersection}):
$$\left(\pi_*j_*b_{|E}^*(x)\right)\cdot\left(\pi_*j_*b_{|E}^*(y)\right)=\frac{1}{2}\left(4\pi_*j_*b_{|E}^{*}(x\cdot y\cdot c_1(A))-4\pi_*b^{*}g_*(x\cdot y) \right).$$
\end{itemize}
\end{proof}
Now we start the proof of Proposition \ref{BasisHilbertH2Main}, we first prove (i) and (vii).
By Theorem \ref{Vois}, we have:
$$H^k(\Bl_{\Delta}(A\times A),\Z)=b^*(H^k(A\times A,\Z)),$$
for $k=1$ and $k=7$. It follows from K\"unneth formula that $\ell_+^1(\Bl_{\Delta}(A\times A))=0$. 
Hence, from Proposition \ref{MainBasicAlphaInequality}, the coefficient of surjectivity $\alpha_1(\Bl_{\Delta}(A\times A))=\alpha_7(\Bl_{\Delta}(A\times A))=0$. It follows (i) and (vii).

Now let us prove (ii) and (vi).
By Theorem \ref{Vois}, we have:
\begin{equation}
H^2(\Bl_{\Delta}(A\times A),\Z)=b^*(H^2(A\times A,\Z))\oplus j_*b_{|E}^*(H^0(\Delta,\Z)).
\label{H2Hilb2bis}
\end{equation}
We have $j_*b_{|E}^*(H^0(\Delta,\Z))=\Z E$. 
From K\"unneth formula, we remark that $\ell_+^2(\Bl_{\Delta}(A\times A))=1$.
Hence, from Proposition \ref{MainBasicAlphaInequality}:
$$\alpha_2(\Bl_{\Delta}(A\times A))+\alpha_6(\Bl_{\Delta}(A\times A))=1.$$
Moreover, we have $\pi_{*}(\mathcal{O}_{\Bl_{\Delta}(A\times A)})=\mathcal{O}_{A^{[2]}}\oplus\mathcal{L}$, with $\mathcal{L}^{2}=\mathcal{O}_{A^{[2]}}(\pi_*(E))$.
So, as it is well known $\pi_*(E)$ is divisible by 2.
Hence, $\alpha_2(\Bl_{\Delta}(A\times A))=1$ and $\alpha_6(\Bl_{\Delta}(A\times A))=0$. 
So, (ii) follows from (\ref{H2Hilb2bis}).
Similarly, by Theorem \ref{Vois}, we have:
$$
H^6(\Bl_{\Delta}(A\times A),\Z)=b^*(H^6(A\times A,\Z))\oplus j_*b_{|E}^*(H^4(\Delta,\Z)),
$$
with $j_*b_{|E}^*(H^4(\Delta,\Z))=\Z f$. So, we obtain (vi) because $\alpha_6(\Bl_{\Delta}(A\times A))=0$. 

It remains to study the most interesting part which is $H^4(A^{[2]},\Z)$. Again by Theorem \ref{Vois}, we have:
\begin{equation}
H^4(\Bl_{\Delta}(A\times A),\Z)=b^*(H^4(A\times A,\Z))\oplus j_*b_{|E}^*(H^2(\Delta,\Z)).
\label{H4Hilb2bis}
\end{equation}
First, we are going to use Remark \ref{remarkdefin2} to show that $\pi_*(b^*(H^4(A\times A,\Z)))$ is primitive in $H^4(A^{[2]},\Z)$. To do so, we verify the hypothesis of Theorem \ref{Main2} for the couple $(A\times A,\mathfrak{S}_2)$.
By K\"unneth formula, we have $\ell_+^{2k+1}(A\times A)=0$ for all $k\in\left\{0,1,2,3\right\}$, $\ell_+^{4}(A\times A)=b_2(A)$ and $\ell_+^{2}(A\times A)=0$. It follows condition (ii) of Theorem \ref{Main2}. Moreover, $\Delta\in H^{4}(A\times A,\Z)$ is primitive because $\Delta\cdot A\otimes pt=1$. 
So, by Theorem \ref{Main2}, $(A\times A,\mathfrak{S}_2)$ is $H^4$-normal and 
by Remark \ref{remarkdefin2} the lattice $\pi_*(b^*(H^4(A\times A,\Z)))$ is primitive in $H^4(A^{[2]},\Z)$.
Moreover, by Corollary \ref{MainBasicCor}, $\log_2\discr \pi_*(b^*(H^4(A\times A,\Z)))=b_2(A)$ and by Lemma \ref{commut} (ii), $\log_2\discr (\pi_*j_*b_{|E}^*(H^2(\Delta,\Z)))=b_2(A)$. 
So:
\begin{equation}
\discr \pi_*(b^*(H^4(A\times A,\Z)))=\discr\pi_*j_*b_{|E}^*(H^2(\Delta,\Z)).
\label{BasisHilb2H4}
\end{equation}
Since $A^{[2]}$ is smooth, we know by (\ref{BasicLatticeTheory2}) that $\discr \pi_*(b^*(H^4(A\times A,\Z)))=\discr \pi_*(b^*(H^4(A\times A,\Z)))^\bot$.
Since $\pi_*j_*b_{|E}^*(H^2(\Delta,\Z))\subset\pi_*(b^*(H^4(A\times A,\Z)))^\bot$, by (\ref{BasisHilb2H4}), $\pi_*j_*b_{|E}^*(H^2(\Delta,\Z))=\pi_*(b^*(H^4(A\times A,\Z)))^\bot$. So, $\pi_*j_*b_{|E}^*(H^2(\Delta,\Z))$ is also primitive in $H^4(A^{[2]},\Z)$. 

Now, we are going to calculate $$\frac{H^4(A^{[2]},\Z)}{\pi_*(b^*(H^4(A\times A,\Z)))\oplus \pi_*j_*b_{|E}^*(H^2(\Delta,\Z))}.$$
For a lattice $L$, we denote by $A_L:=L^\vee/L$ its discriminant group and for $x\in L^\vee$, we denote by $\overline{x}$ its reduction in $A_L$.
Let denote $N:=\pi_*(b^*(H^4(A\times A,\Z)))$ and $N^{\bot}=\pi_*j_*b_{|E}^*(H^2(\Delta,\Z))$. 
By Lemma \ref{commut} (ii), $\frac{1}{2}N^{\bot}= (N^{\bot})^\vee$. It follows that: 
\begin{equation}
A_{N^{\bot}}\simeq \frac{\frac{1}{2}N^{\bot}}{N^{\bot}}=(\Z/2\Z)^{b_2(A)}.
\label{BasisHilb2Equation1}
\end{equation}
Since $H^4(A^{[2]},\Z)$ is unimodular, by (\ref{BasicLatticeTheory3}), we have an isometry between the discriminant groups:
$$\nu:A_{N}\simeq A_{N^{\bot}}.$$
Again by Lemma \ref{commut} (ii), $\frac{1}{2}\pi_*b^{*}(x\otimes x)\in N^\vee$, 
for all $x\in H^2(A,\Z)$.
So, $A_{N}$ is generated by the classes $\overline{\frac{1}{2}\pi_*b^{*}(x\otimes x)}$, where $x\in H^2(A,\Z)$.
Let $\left\{\left.z(x)\in N^{\bot}\right|\ x\in H^{2}(A,\Z)\right\}$ be a basis of $N^{\bot}$ such that $\nu\left(\overline{\frac{\pi_*b^{*}(x\otimes x)}{2}}\right)=\overline{\frac{z(x)}{2}}$.
Since $\nu$ is an isometry, we have: 
\begin{equation}
\frac{\pi_*b^{*}(x\otimes x)}{2}\cdot\frac{\pi_*b^{*}(y\otimes y)}{2}=\frac{z(x)}{2}\cdot\frac{z(y)}{2}\mod \Z.
\label{BasisHilb2Equation0}
\end{equation}
However, by Lemma \ref{IntersectionHilb2} (i): 
$$\pi_*b^{*}(x\otimes x)\cdot\pi_*b^{*}(y\otimes y)=2(x\cdot y)^2.$$
And by Lemma \ref{IntersectionHilb2} (iv):
$$\pi_*j_*b_{|E}^*(x)\cdot \pi_*j_*b_{|E}^*(y)=2x\cdot y.$$
Since $(x\cdot y)=(x\cdot y)^2 \mod 2$, it follows from (\ref{BasisHilb2Equation0}) that:
$$\frac{z(x)}{2}\cdot\frac{z(y)}{2}=\frac{\pi_*j_*b_{|E}^*(x)}{2}\cdot \frac{\pi_*j_*b_{|E}^*(y)}{2}\mod \Z.$$
So, we have constructed an isometry of $A_{N^{\bot}}$ sending $\overline{\frac{z(x)}{2}}$ to $\overline{\frac{\pi_*j_*b_{|E}^*(x)}{2}}$.
However, \cite[Theorem 3.6.3]{Lattice} says that the natural map $\mathcal{O}(N^\bot)\rightarrow \mathcal{O}(A_{N^{\bot}})$ is surjective 
when $N^{\bot}$ is an indefinite lattice which is always the case apart when $A$ is the projective space; but in this cases , $\rk  N^\bot=1$ and the natural map $\mathcal{O}(N^\bot)\rightarrow \mathcal{O}(A_{N^\bot})$ is obviously surjective.

It follows that there is an isometry $\widetilde{\theta}$ of $N^{\bot}$ such that:
$$\widetilde{\theta}(\pi_*j_*b_{|E}^*(x))=z(x)\mod 2N^\bot.$$
Then, $\widetilde{\theta}$ provides an isometry $\theta$ on $H^2(A,\Z)$ such that:
$$\pi_*j_*b_{|E}^*(\theta(x))=z(x)\mod 2N^\bot.$$
Then, by definition of $\nu$, the elements 
$$\pi_*j_*b_{|E}^*(\theta(x))+\pi_*b^{*}(x\otimes x)$$ are divisible by 2 for all $x\in H^2(A,\Z)$.
\subsection{Symplectic groups and Beauville--Bogomolov forms}
We will study the Beauville--Bogomolov form of some quotients of an irreducible symplectic $2n$-fold $X$ by a symplectic group $G$ of prime order.
So, $M:=X/G$ will be a singular primitively symplectic orbifold. 
One of our most important ingredients will be the Fujiki formula generalized by Matsushita in \cite{Mat} (see \cite[Section 1.2.2]{Lol} for an overview). 
Assume that $\codim \sing M\geq 4$, then if $Y=M$ or $X$, we have for all $\alpha\in H^2(Y,\Z)$:
\begin{equation}
\alpha^{2n}=c_Y B_Y(\alpha,\alpha)^n,
\label{Fujiki}
\end{equation}
where $c_Y$ is the Fujiki constant. Moreover, if $\omega\neq0$ is the holomorphic 2-form of $Y$, the convention is:
\begin{equation}
B_Y(\omega+\overline{\omega},\omega+\overline{\omega})>0.
\label{FujikiConvention}
\end{equation}
Furthermore, the convention is to choose the Beauville--Bogomolov form to be integral and indivisible.  

We know the Beauville--Bogomolov form $B_X$ of $X$. Then, we will deduce the Beauville--Bogomolov form $B_M$ of $M$ from $B_X$ and $H^2(M,\Z)$.
The link between $B_X$ and $B_M$ is the matter of the following proposition.
\begin{prop}\label{beau}
Let $X$ be an irreducible symplectic manifold of dimension $2n$ and $G$ be a symplectic group of prime order $p$ with $\codim\Fix G\geq 4$.
Then,
$$B_{M}(\pi_{*}(\alpha),\pi_{*}(\beta))=\sqrt[n]{\frac{c_Xp^{2n-1}}{c_{M}}}B_{X}(\alpha,\beta),$$
where $c_{M}$ is the Fujiki constant of $M$ and $\alpha$, $\beta$ are in $H^{2}(X,\Z)^{G}$.
\end{prop}
\begin{proof}
By (\ref{Fujiki}), we have: 
\begin{equation}
(\pi_{*}(\alpha))^{2n}=c_{M}B_{M}(\pi_{*}(\alpha),\pi_{*}(\alpha))^{n},
\label{jenesaisplus}
\end{equation}
and:
\begin{equation}
\alpha^{2n}=c_XB_{X}(\alpha,\alpha)^{n}.
\label{jenesaisplus2}
\end{equation}
Moreover, by Lemma \ref{commut} (ii):
\begin{equation}
(\pi_{*}(\alpha))^{2n}=p^{2n-1}\alpha^{2n}.
\label{jenesaisplus3} 
\end{equation}
So, by (\ref{jenesaisplus}), (\ref{jenesaisplus2}), (\ref{jenesaisplus3}), we have:
$$B_{M}(\pi_{*}(\alpha),\pi_{*}(\alpha))^{n}=\frac{c_Xp^{2n-1}B_{X}(\alpha,\alpha)^{n}}{c_M}.$$
By (\ref{FujikiConvention}), we get the result.
\end{proof}
When we assume that $X$ is a manifold of $K3^{[2]}$-type, from \cite[Section 9]{Beauville}, we know that the Fujiki constant of $X$ is 3.
So for $\alpha$, $\beta$ in $H^{2}(X,\Z)^{G}$:
\begin{equation}
B_{M}(\pi_{*}(\alpha),\pi_{*}(\beta))=\sqrt{\frac{3p^{3}}{c_{M}}}B_{X}(\alpha,\beta).
\label{beau2}
\end{equation}

We need a variation of Lemma \ref{latima} in the case of the lattice of $H^{2}(X,\Z)$ endowed with the Beauville--Bogomolov form.
\begin{lemme}\label{lemfin}
Let $X$ be an irreducible symplectic manifold of $K3^{[2]}$-type and $G=\left\langle g\right\rangle$ be a symplectic automorphism group of prime order $p\geq 3$.
We have:
\begin{itemize}
\item[(i)]
$A_{H^{2}(X,\Z)^{G}}=(\Z/2\Z)\oplus (\Z/p\Z)^{\ell_{p}^2(X)}$,
\item[(ii)]
$\discr H^{2}(X,\Z)^{G}=2p^{\ell_{p}^2(X)}$.
\item[(iii)]
We denote $A_{H^{2}(X,\Z)^{G},p}:=(\Z/p\Z)^{\ell_{p}^2(X)}$.
Then, the projection $$\frac{H^{2}(X,\Z)}{H^{2}(X,\Z)^{G}\oplus \ker \sigma}\rightarrow A_{H^{2}(X,\Z)^{G},p}$$ is an isomorphism, 
where $\sigma=\id+g+...+g^{p-1}$.
\item[(iv)]
Moreover, let $x\in H^{2}(X,\Z)^{G}$. We have $\frac{x}{p}\in (H^{2}(X,\Z)^{G})^{\vee}$ if and only if there is $z\in H^{2}(X,\Z)$ such that $x=z+g(z)+...+g^{p-1}(z)$. 
\item[(v)]
Also:
$$\frac{\pi_*(H^{2}(X,\Z))}{\pi_*(H^{2}(X,\Z)^G)}=(\Z/p\Z)^{\ell_{p}^2(X)}.$$
\end{itemize}
\end{lemme}
\begin{proof}
We recall that $p\leq 11$ by \cite{Mong2}.
Thus, the first three assumptions follow from \cite[Lemma 6.5]{SmithTh} and its proof (the number $a_G(X)$ in \cite[Lemma 6.5]{SmithTh} equal $\ell_{p}^2(X)$ from \cite[Corollary 5.8]{SmithTh}).
Then, the proof of (iv) is identical to the proof of Lemma \ref{latima}.

It remains to prove (v). 
Let $$\mathscr{A}:=\left\{\left.\frac{z+...+g^{p-1}(z)}{p}\in (H^{2}(X,\Z)^G)^\vee\right|\ z\in H^{2}(X,\Z)\right\}.$$
By (iv): 
$$\mathscr{A}=\left\{\left.\frac{x}{p}\in (H^{2}(X,\Z)^G)^\vee\right|\ x\in H^{2}(X,\Z)^G\right\}.$$
So:
$$\frac{\mathscr{A}}{H^{2}(X,\Z)^G}=A_{H^{2}(X,\Z)^{G},p}.$$
The following morphism
$$\xymatrix@R=0pt{f:\mathscr{A}\ \ \ar[r] & \frac{\pi_*(H^{2}(X,\Z))}{\pi_*(H^{2}(X,\Z)^G)}\\
\ \ \ \ \ \ \ \ \ \frac{z+...+g^{p-1}(z)}{p}\ar[r]& \overline{\pi_*(z)}}$$
is surjective by definition. 
Quotiented $\mathscr{A}$ by $H^{2}(X,\Z)^G$, $f$
provides an isomorphism between $A_{H^{2}(X,\Z)^{G},p}$ and $\frac{\pi_*(H^{2}(X,\Z))}{\pi_*(H^{2}(X,\Z)^G)}$.
Then, the result follows from (iii).
\end{proof}
\subsection{Quotient of a $K3^{[2]}$-type manifold by an automorphism of order 11}\label{SymplecticOrder11}
As mentioned in the introduction, there are two different examples of symplectic automorphisms of order 11 on a manifold of $K3^{[2]}$-type (Example 4.5.1 and Example 4.5.2 in \cite{MongT}).
In both examples, the fixed locus is a set of 5 isolated points. But as it is explained in \cite[Section 7.4.4]{MongT}, the lattice $H^{2}(X,\Z)^{G}$ will be different in these cases.
In the first case the lattice is $\begin{pmatrix} 6 & 2 & 2 \\ 2 & 8 & -3\\ 2 & -3 & 8\end{pmatrix}$, and in the second $\begin{pmatrix} 2 & 1 & 0 \\ 1 & 6 & 0\\ 0 & 0 & 22\end{pmatrix}$.
Moreover, it is proved in \cite[Section 7.4.4]{MongT} that these are the only possible invariant lattices for $p=11$.
So, let $X$ be a manifold of $K3^{[2]}$-type and $G$ a symplectic automorphism group of $X$ of order 11.
If $H^{2}(X,\Z)^{G}$ is isomorphic to $\begin{pmatrix} 6 & 2 & 2 \\ 2 & 8 & -3\\ 2 & -3 & 8\end{pmatrix}$, we will denote the quotient $X/G$ by $M_{11}^{1}$, 
and if $H^{2}(X,\Z)^{G}$ is isomorphic to $\begin{pmatrix} 2 & 1 & 0 \\ 1 & 6 & 0\\ 0 & 0 & 22\end{pmatrix}$, we will denote the quotient $X/G$ by $M_{11}^{2}$.
Hence, we obtain the following corollary.
\begin{cor}\label{p11}
Let $X$ be a manifold of $K3^{[2]}$-type. Let $G$ be a symplectic automorphism group of $X$ of order 11.
Then, $(X,G)$ is $H^{2}$-normal and $H^{4}$-normal.
\end{cor}
\begin{proof}
In the both cases, we have $\discr H^{2}(X,\Z)^{G}=2\times11^{2}$. Hence, by Lemma \ref{lemfin} (ii):
\begin{equation}
\ell_{11}^2(X)=2.
\label{ApplicationEquation}
\end{equation}
It follows from Corollary \ref{BasicLatticeQuotient} that 
$$\ell_+^2(X)=1\ \text{and}\ \ell_-^2(X)=0.$$
We recall from Section \ref{BoissiereSartiH2}, that when $p\geq3$, we have $\ell_+^k(X)=\ell_1^k(X)$ and $\ell_-^k(X)=\ell_{p-1}^k(X)$ for all $k$.
Hence, by (\ref{Sym2}) and \cite[Lemma 6.14 ]{SmithTh}:
$$\ell_+^4(X)=1.$$
So, it follows from Remark \ref{MainBasicAlphaInequalityRemark} that:
$$\alpha_4(X)\leq \frac{1}{2}.$$
Since $\alpha_4(X)$ is an integer, we have $\alpha_4(X)=0$. So, $(X,G)$ is $H^4$-normal.
Therefore, the $H^2$-normality follows from Corollary \ref{H2H4}.
\end{proof}

The end of the section is dedicated to the proof of Theorem \ref{MainOrder11}.

Assume
$$H^{2}(X,\Z)^{G}\simeq\begin{pmatrix} 6 & 2 & 2 \\ 2 & 8 & -3\\ 2 & -3 & 8\end{pmatrix}.$$
We denote by $\left\{a, b, c\right\}$ an integral basis of $H^{2}(X,\Z)^{G}$ with $B_{X}(a,a)=6$, $B_{X}(b,b)=8$, $B_{X}(c,c)=8$, $B_{X}(a,b)=B_{X}(a,c)=2$ and $B_{X}(b,c)=-3$. Let 
$$\mathcal{B}=\left\{\frac{\pi_{*}(b)-\pi_{*}(c)}{11}, \frac{4\pi_{*}(a)-\pi_{*}(b)}{11}, \frac{\pi_{*}(a)-3\pi_{*}(c)}{11}\right\}.$$
We show that $\mathcal{B}$ is a basis of $H^{2}(M_{11}^{1},\Z)$.
By Lemma \ref{lemfin} (iv),  $\mathcal{B}\subset \pi_*(H^{2}(X,\Z))$. 
Moreover, we have $$\frac{\left\langle \mathcal{B}\right\rangle}{\pi_{*}(H^{2}(X,\Z)^{G})}=(\Z/11\Z)^2.$$
Hence, by Lemma \ref{lemfin} (v) and (\ref{ApplicationEquation}):
$$\left\langle \mathcal{B}\right\rangle=\pi_{*}(H^{2}(X,\Z)).$$
It follows from Corollary \ref{p11} that $\mathcal{B}$ is a basis of $H^{2}(M_{11}^{1},\Z)$.

The matrix of the sublattice of $H^{2}(X,\Z)^{G}$ generated by $b-c$, $4a-b$ and $a-3c$ is
$$\begin{pmatrix} 2\times11 & -11 & 3\times 11\\ -11& 8\times 11& -11\\ 3\times11 & -11 & 6\times11\end{pmatrix}.$$
Then, by (\ref{beau2}), the Beauville--Bogomolov form on $H^{2}(M_{11}^{1},\Z)$ gives the lattice
$$\frac{1}{11}\sqrt{\frac{3\cdot11^{3}}{c_{M_{11}^{1}}}}\begin{pmatrix} 2 & -1 & 3 \\ -1 & 8 & -1\\ 3 & -1 & 6\end{pmatrix}.$$
Since the Beauville--Bogomolov form is integral and indivisible, it follows that $c_{M_{11}^{1}}=33$, and we get the lattice
$$\begin{pmatrix} 2 & -1 & 3 \\ -1 & 8 & -1\\ 3 & -1 & 6\end{pmatrix}.$$

Now assume that $$H^{2}(X,\Z)^{G}\simeq\begin{pmatrix} 2 & 1 & 0 \\ 1 & 6 & 0\\ 0 & 0 & 22\end{pmatrix}.$$
The proof is identical taking the basis $$\mathcal{B}=\left\{\frac{\pi_{*}(a)-2\pi_{*}(b)}{11},\frac{6\pi_{*}(a)-\pi_{*}(b)}{11}, \frac{\pi_{*}(c)}{11}\right\},$$
where $\left\{a, b, c \right\}$ is an integral basis of $H^{2}(X,\Z)^{G}$ with $B_{X}(a,a)=2$, $B_{X}(b,b)=6$, $B_{X}(c,c)=22$, $B_{X}(b,c)=B_{X}(a,c)=0$ and $B_{X}(a,b)=1$. 
\subsection{Quotient of a $K3^{[2]}$-type manifold by a symplectic automorphisms of order 3}\label{SymplecticOrder3}
As mentioned in the introduction, in \cite[Theorem 7.2.7]{MongT}, Mongardi distinguishes two kinds of symplectic automorphism of order 3 on a $K3^{[2]}$-type manifold one with 27 isolated fixed points and another with a fixed abelian surface. We recall from \cite[Example 4.2.1]{MongT}:
\begin{rmk}
Let $S$ be a K3 surface.
A natural symplectic automorphism of order 3 on $S^{[2]}$ has 27 isolated fixed points.
\end{rmk}
Moreover, Mongardi in \cite[Theorem 2.5]{Mong} shows the following result:
\begin{prop}\label{MongEtReMong}
Let $X$ be a $K3^{[2]}$-type manifold and $\varphi$ be a symplectic automorphism of order 3 with 27 isolated fixed points. Then,  
 $(X,\varphi)$ can be deformed to a couple $(S^{[2]},\phi^{[2]})$, where $S$ is a K3 surface and $\phi^{[2]}$ is the automorphism induced on $S^{[2]}$ by a symplectic automorphism $\phi$ of order 3 on $S$.
 \end{prop}
In this section, we are considering these symplectic automorphisms $\varphi$ of order 3, with 27 isolated points, on a $K3^{[2]}$-type manifold $X$; and we denote the quotient $X/\varphi$ by $M_3$. 
\begin{cor}\label{p3}
Let $X$ be a manifold of $K3^{[2]}$-type. Let $G$ be a symplectic automorphisms group of order 3 of $X$ with 27 isolated fixed points.
Then, $(X,G)$ is $H^{2}$-normal and $H^{4}$-normal.
\end{cor}
\begin{proof}
By Corollary \ref{H2H4},
if $(X,G)$ is $H^{4}$-normal then $(X,G)$ is $H^{2}$-normal.
Hence, we have only to show the $H^{4}$-normality. 

The action of $G$ is symplectic on $X$, in this case, the only possible local action around the fixed points is given by: $$\diag(\xi_3,\xi_3,\xi_3^2,\xi_3^2),$$ where  $\xi_3$ is a 3-root of the unity. So, unfortunately, the fixed points of $G$ are not simple. However, since $G$ is of order 3, we can get around this difficulty using Proposition \ref{blowup2}. 

Let $X_1$ be the blow-up of $X$ in $\Fix G$ and $G_1$ the automorphisms group induced from $G$. From Lemma \ref{ResolutionLemma} the fixed points of $G_1$ are simple. Hence, it will be possible to apply Theorem \ref{main} to $(X_1,G_1)$. 

First, using the same method as in the proof of Lemma \ref{ResolutionLemma}, we remark that for each isolated fixed point of $G$, we get 2 rational fixed lines of $G_1$. So, $\Fix G_1$ consists in 54 rational lines. 
Moreover, from \cite{Markman}, $H^{*}(X,\Z)$ is torsion-free, then from Theorem \ref{Vois}, $H^{*}(X_1,\Z)$ is also torsion-free.
Hence, it only remains to verify the condition (ii) of Theorem \ref{main}.

From Theorem \ref{Vois}:
$$H^{i}(X_1,\Z)\simeq H^{i}(X,\Z)\oplus H^0(pt,\Z)^{27\oplus}\ \text{and}\ H^{j}(X_1,\Z)\simeq H^{j}(X,\Z),$$
for $i\in\left\{2,4,6\right\}$ and $j\in\left\{0,1,3,5,7,8\right\}$.
Since the exceptional divisors are fixed by the action of $G_1$, it follows that:
\begin{equation} 
\ell_+^i(X_1)=\ell_+^i(X)+27,
\label{Proofp=3}
\end{equation}
for all $i\in\left\{2,4,6\right\}$. The others $\ell_*^j$ are the same for $X$ and $X_1$. 
Moreover, from \cite[Theorem 1.1]{SmithTh} and Proposition \ref{BNSFormula}:
$$h^*(\Fix G,\Q)=\sum_{1\leq q < p} \ell_{q}^*(X).$$
Since $\Fix G$ consists in 27 isolated points, we get:
\begin{equation}
27=\sum_{1\leq q < p} \ell_{q}^*(X).
\label{Proofp=3a}
\end{equation}
By (\ref{Proofp=3}):
\begin{equation}
\sum_{1\leq q < p} \ell_{q}^*(X_1)=\sum_{1\leq q < p} \ell_{q}^*(X)+3\times27.
\label{Proofp=3b}
\end{equation}
Since $\Fix G_1$ consists in 54 rational curves, we have:
\begin{equation}
h^*(\Fix G_1,\Q)=2\times54=4\times 27.
\label{Proofp=3c}
\end{equation}
Finally, from (\ref{Proofp=3a}), (\ref{Proofp=3b}) and (\ref{Proofp=3c}), we get:
$$h^*(\Fix G_1,\Q)=\sum_{1\leq q < p} \ell_{q}^*(X_1).$$
It follows from Theorem \ref{main} that $(X_1,G_1)$ is $H^4$-normal.
We conclude with Proposition \ref{parrynormal}.
\end{proof}
The end of the section is dedicated to the proof of Theorem \ref{MainOrder3}.

From \cite[Theorem 4.1]{Sati} and Proposition \ref{MongEtReMong}, there is an isometry of lattices $H^{2}(X,\Z)^{G}\simeq U\oplus U(3)\oplus U(3)\oplus A_{2}\oplus A_{2}\oplus (-2)$.
In the rest of the proof, we identify $H^{2}(X,\Z)^{G}$ with the lattice $U\oplus U(3)\oplus U(3)\oplus A_{2}\oplus A_{2}\oplus (-2)$

By Lemma \ref{lemfin} (iv), we have:
\begin{equation}
\frac{1}{3}\pi_{*}(U(3))\subset H^{2}(M_{3},\Z).
\label{LastEquations}
\end{equation}
Let $(a, b)$ an integral basis of $A_{2}$, with $B_{X}(a,a)= B_{X}(b,b)=-2$ and $B_{X}(a,b)=1$.
Idem, by Lemma \ref{lemfin} (iv), we have: 
$$\frac{\pi_*(a)-\pi_*(b)}{3}\in H^{2}(M_{3},\Z).$$
The lattice generated by $a-b$ and $a+2b$ is isomorphic to $A_2(3)=\begin{pmatrix} -6 & 3 \\ 3 & -6\end{pmatrix}$. So, we denote $A_2(3):=\left\langle a-b,a+2b\right\rangle$.
We have: 
\begin{equation}
\frac{1}{3}\pi_*(A_2(3))\in H^{2}(M_{3},\Z).
\label{LastEquations2}
\end{equation}
Then, by (\ref{LastEquations}) and (\ref{LastEquations2}):
$$\pi_*(H^{2}(X,\Z))\supset\pi_{*}\left(U\oplus\frac{1}{3}U(3)^{2}\oplus\frac{1}{3}A_{2}(3)^{2}\oplus(-2)\right).$$
Therefore, by (ii) and (v) of Lemma \ref{lemfin}, we obtain:
$$\pi_*(H^{2}(X,\Z))=\pi_{*}\left(U\oplus\frac{1}{3}U(3)^{2}\oplus\frac{1}{3}A_{2}(3)^{2}\oplus(-2)\right).$$
So by Corollary \ref{p3}, 
$$H^2(M_{3},\Z)=\pi_{*}\left(U\oplus\frac{1}{3}U(3)^{2}\oplus\frac{1}{3}A_{2}(3)^{2}\oplus(-2)\right).$$
Then, by (\ref{beau2}), the Beauville--Bogomolov form of $H^{2}(M_{3},\Z)$ gives the lattice
\begin{align*}
&U\left(\sqrt{\frac{81}{c_{M_{3}}}}\right)\oplus\frac{1}{3}U^{2}\left(3\sqrt{\frac{81}{c_{M_{3}}}}\right)\oplus\frac{1}{3}A_{2}^{2}\left(3\sqrt{\frac{81}{c_{M_{3}}}}\right)\oplus\left(-2\sqrt{\frac{81}{c_{M_{3}}}}\right)\\
&=U\left(3\sqrt{\frac{9}{c_{M_{3}}}}\right)\oplus U^{2}\left(\sqrt{\frac{9}{c_{M_{3}}}}\right)\oplus A_{2}^{2}\left(\sqrt{\frac{9}{c_{M_{3}}}}\right)\oplus\left(-6\sqrt{\frac{9}{c_{M_{3}}}}\right).
\end{align*}
Then knowing that the Beauville--Bogomolov form is integral and indivisible, we have $c_{M_{3}}=9$ and we get the lattice
$$U(3)\oplus U^{2}\oplus A_{2}^{2}\oplus (-6).$$
\bibliographystyle{amssort}

\noindent
Gr\'egoire \textsc{Menet}

\noindent
IMB, Universit\'e de Dijon

\noindent 
9 avenue Alain Savary, 21000 DIJON (France)

\noindent
{\tt gregoire.menet@u-bourgogne.fr}

\end{document}